\documentclass{elsarticle}

\usepackage[english]{babel}
\usepackage{amsfonts,amsmath,amssymb,amsthm}
\usepackage{nicefrac}
\usepackage{url}
\usepackage{ifthen}
\usepackage{verbatim}
\usepackage{mathrsfs}
\usepackage{textcomp}
\usepackage{ifpdf}
\usepackage{color}
\usepackage{aeguill}

\usepackage{tikz,subfigure}
\usetikzlibrary{snakes,calc}

\newcommand{\lab}[1]{\label{#1}}                

\newcommand{\myeqref}[1]{\eqref{#1}}  

\newcommand\eqn[1]{(\ref{#1})}

\newcommand{\be}{\begin{equation}}
  \newcommand{\ee}{\end{equation}}
\newcommand{\bea}{\begin{eqnarray}}
  \newcommand{\eea}{\end{eqnarray}}
\newcommand{\bean}{\begin{eqnarray*}}
  \newcommand{\eean}{\end{eqnarray*}}

\newtheorem{thm}{Theorem}
\newtheorem{cor}[thm]{Corollary}

\newtheorem{lemma}[thm]{Lemma}

\newtheorem{prop}[thm]{Proposition}

\def\G{{\mathcal G}}

\def\ex{{\bf E}}

\def\N{{\mathcal N}}
\def\S{{\mathcal S}}

\def\eps{\epsilon}
\def\la{\lambda}
\def\ss{\smallskip}

\catcode`@=11 \@addtoreset{equation}{section}

\catcode`@=12

\newcommand{\setN}{\ensuremath{\mathbb{N}}}

\newcommand{\eqdef}{:=}
\newcommand{\eqdefinv}{=:}

\newcommand{\set}[2][]{#1\{ {#2} #1\}}
\newcommand{\paren}[2][]{#1( {#2} #1)}
\newcommand{\card}[2][]{#1| #2 #1|}
\newcommand{\abs}[2][]{#1| #2 #1|}

\newcommand{\floor}[2][]{#1\lfloor #2 #1\rfloor}

\newcommand{\mean}[2][]{\mathbf{E}\,#1(#2#1)}

\newcommand{\var}[2][]{\mathop{\rm Var}#1(#2#1)}

\newcommand{\Mcal}{\ensuremath{\mathcal{M}}}
\newcommand{\Ccal}{\ensuremath{\mathcal{C}}}
\newcommand{\Kcal}{\ensuremath{\mathcal{K}}}

\newcommand{\Tcal}{\ensuremath{\mathcal{T}}}
\newcommand{\Xcal}{\ensuremath{\mathcal{X}}}
\newcommand{\Scal}{\ensuremath{\mathcal{S}}}

\newcommand{\iseq}{\ensuremath{\mathbf{i}}}
\newcommand{\rseq}{\ensuremath{\mathbf{r}}}

\newcommand{\Gnp}{\G(n,p)}
\newcommand{\cn}{\check n}
\newcommand{\cm}{\check m}
\newcommand{\cp}{\check p}
\newcommand{\HM}{H_{\Mcal}}

\allowdisplaybreaks

\title{A transition of limiting distributions of large matchings in
  random graphs}
\author[gao]{Pu Gao\fnref{fn1}}
\ead{pu.gao@utoronto.ca}

\author[sato]{Cristiane M.~Sato\fnref{fn2}}
\ead{cmsato@uwaterloo.ca}

\address[gao]{University of Toronto}
\address[sato]{Universidade Federal do ABC, University of Waterloo}

\fntext[fn1]{Research supported by the NSERC PDF.}
\fntext[fn2]{Partially supported by FAPESP (Proc.2103/03447-6)}
\begin{document}

\begin{abstract}
  We study the asymptotic distribution of the number of matchings of
  size $\ell=\ell(n)$ in $\Gnp$ for a wide range of $p=p(n)\in (0,1)$
  and for every $1\le \ell\le \floor{n/2}$. We prove that this
  distribution changes from normal to log-normal as $\ell$ increases,
  and we determine the critical value of $\ell$, as a function of $n$
  and $p$, at which the transition of the limiting distribution
  occurs.
\end{abstract}

\maketitle

\section{Introduction}

Let $\Gnp$ denote the probability space of random graphs on $n$
vertices, where each edge is included independently with probability
$p$. A classical result by Ruci\'{n}ski~\cite{R} shows that the
distribution of the number of small subgraphs (meaning the number of
subgraphs isomorphic to a graph with a fixed size) is asymptotically
normal if its expected value goes to infinity as $n$ goes to infinity.
This is naturally expected as this random variable can be expressed as
a sum of indicator random variables such that each variable is
dependent only on a small proportion of the other variables. However,
this intuitive explanation fails when the size of the subgraphs
increases since then each indicator variable depends on more and more
of the other variables. It has been shown by Janson~\cite{J3} that the
numbers of spanning trees, perfect matchings, and Hamilton cycles in
$\Gnp$ (when $p$ is in an appropriate range) are asymptotically
log-normally distributed, which behave quite differently from
variables with the normal distribution. The first author~\cite{G4,G6}
recently proved that the numbers of $d$-factors (for $d$ not growing
too fast), triangle-factors and triangle-free subgraphs also follow a
log-normal distribution (when $p$ is in an appropriate range).
Comparing the result by Ruci\'{n}ski~\cite{R} with that by
Janson~\cite{J3}, we notice that the distribution of the number of
$\ell$-matchings (matchings of size $\ell$) must undergo certain
phases of transition, starting from normal and ending with log-normal,
when $\ell$ increases from a constant size to $\floor{n/2}$. This
motivates our research in this paper. We study the asymptotic
distribution of the number of matchings of size $\ell$ in $\Gnp$,
denoted by $X_{n,\ell}$, for every $1\le \ell\le \floor{n/2}$. In
particular, we prove that $X_{n,\ell}$ is asymptotically normal if
$\ell=o(n\sqrt{p})$ and is asymptotically log-normal if
$\ell=\Omega(n\sqrt{p})$. This holds for all $p$ such that
$1-p=\Omega(1)$ and $n^{1/8-\eps} p\to\infty$, where $\eps>0$ is an
arbitrarily small constant. To our best knowledge, this is the first
paper that studies the distribution of the number of copies of a
subgraph whose order is between constant and $n$ in
$\Gnp$.

This same phenomenon of the transition of limiting distributions of a
certain subgraph count as the size of the subgraph increases has been
observed and studied in another well-known random graph space: the
random $d$-regular graphs. There is a classical result by
Bollob\'{a}s~\cite{B} and Wormald~\cite{W} stating that the
distributions of the numbers of short cycles (cycles with constant
sizes) in a random $d$-regular graph are asymptotically Poisson, known
as the Poisson paradigm~\cite{AS}, whereas it was observed later by
Robinson and Wormald~\cite{RW4,RW5} that the number of Hamilton cycles
is determined by the numbers of short cycles. Janson~\cite{J} proved
that the logarithm of the number of Hamilton cycles can be expressed
as the linear combination of a sequence of independent Poisson
variables, based on the results in~\cite{RW4,RW5}. Garmo~\cite{G}
filled the gap and determined the distribution of all long
cycles, whose sizes vary from constant to $n$ (i.e.\ the Hamiltonian
cycles). His result also describes the critical point (of the size of
the cycles), at which the distribution of the number of the cycles
changes from a linear combination of independent Poisson variables to
the exponential of that form, the same as what was described
in~\cite{J}.

Note that the proof of our main theorem is not just a generalisation
of the proofs in~\cite{R,J3,G4,G6}. In fact, we use very different
approaches and new techniques. We do apply basic tools that also
appear in~\cite{R,G6} to show that a sequence of distributions
converges to normal or log-normal.  Our proof consists of three
parts. In the first part, we study the subcritical case, where
$\ell=o(n\sqrt{p})$.  The second part deals with $\ell$ such that
$\ell=\Omega(n\sqrt{p})$ but $\ell$ is not too close to $n/2$, whereas
the last part focuses on the near-perfect matchings, where $\ell$ is
very close to $n/2$ (i.e.\ $\ell=n/2-O(n^{\alpha})$ for some
$0<\alpha<1$).  The proof techniques and tools used in these three
parts are different. In the first part, we will use the method of
moments~\cite{JLR} to show that the distribution of $X_{n,\ell}$ is
asymptotically normal. This same method was also used in
Ruci\'{n}ski's proof for constant $\ell$. However, the method of
moments cannot be used to prove distributions that are not uniquely
determined by its moments, for instance, the log-normal
distribution. (For more details on the problem of moments, we refer
the reader to~\cite{B4}.) For this reason, we will use another theorem
from~\cite{G6}, known as the log-normal paradigm, as a basic tool to
prove the second and third parts, equipped with the switching method
(described below). The proof for the third part is a generalisation of
the proof in~\cite{G4} for the perfect matchings, whereas the
switchings used in the second part are very different.

 The switching method was first introduced by McKay~\cite{M2} to
 enumerate (sparse) graphs with given degree sequences. In general,
 the method defines a set of switching operations that map graphs in a
 set $A$ to graphs in another set $B$. By computing the number of
 switchings from $A$ to $B$ and the inverse switchings from $B$ to
 $A$, we can estimate the ratio $|A|/|B|$ in some cases quite
 precisely. This method has been widely used to estimate the
 probability that a multigraph generated by the configuration
 model~\cite{B6} is simple (e.g.\ see~\cite{MW,MW2,GSW2}) or to
 estimate probabilities of certain events (e.g.\ see~\cite{PP}).

 Applying the log-normal paradigm in~\cite[Theorems 1 and 3]{G6}
 requires a close analysis of the set of ordered pairs of
 $\ell$-matchings $(M_1,M_2)$ that share exactly $j$ edges, for $j$ in
 a certain range. The analysis in~\cite{G4} (for perfect matchings) is
 based on an argument using the switching method and that proof easily
 extends to our proof for the near-perfect matchings.  Compared with
 the case of the near-perfect matchings, the difference in the proof
 for the second part (for large $\ell$ but not near-perfect matchings)
 lies in the additional effort to analyse the typical values of the
 number of vertices incident to both matchings in the pair
 $(M_1,M_2)$.

 In the proof for the subcritical case $\ell=o(n\sqrt{p})$, in order to
 apply the method of moments, we need to compute the $k$-th central
 moment for each integer $k\ge 1$, which requires a close study of the
 graph structure composed by the union of $k$ (not necessarily
 distinct) $\ell$-matchings. With an unusual use of the switching
 method (unlike in~\cite{M2,MW,MW2} and most other work that uses
   the switching method, in which usually a small number of edges are
   switched, we may switch $o(n)$ edges in a single step), we will
 characterise the graph structure that leads the contribution to the
 $k$-th central moment. As shown in Lemma~\ref{l:sub-significant} in
 Section~\ref{s:sub-significant}, for each even $k$, the leading
 structure is $k/2$ edge-disjoint kissing pairs; whereas for odd $k$,
 the leading structure is $(k-3)/2$ edge-disjoint kissing pairs
 together with a chained triple or a flower with $3$ petals. (The
 terminology of kissing pairs, chained triples and flowers are defined
 in Section~\ref{s:sub-significant} and an example is given in
 Figure~\ref{fig:structures}.) We think this is the first time that
 the switching method is used to determine certain graph
 structures. These leading structures were proved by Ruci\'{n}ski for
 constant $\ell$ (with a different approach), but the use of the
 switchings allows us to derive a proof for all $\ell=o(n\sqrt{p})$.

 In this paper, we assume that $1-p=\Omega(1)$ and $p\ge
 n^{-1/8+\eps}$ for some small constant $\eps>0$. In fact, we only
 assume $1-p=\Omega(1)$ and $p=\omega(n^{-2})$ for the
 subcritical case. The case where $1-p\to 0$ is less interesting as
 there is less ``randomness'', and this condition is indeed necessary
 for the supercritical case since, otherwise, the limiting
 distribution of the number of perfect matchings (assuming that $n$ is
 even) will be normal instead of log-normal (see~\cite[Theorem
 2.3]{G4}). The case $p=O(n^{-2})$ is also less interesting as in this
 case there will be bounded number of edges present, pairwise
 vertex-disjoint, with probability going to $1$. The asymptotic
 distribution function of $X_{n,\ell}$ can be explicitly formulated
 and it is easy to see that $X_{n,\ell}$ is not Poisson convergent
 unless $\ell=1$. This agrees with the result by
 Ruci\'{n}ski~\cite[Theorem 1]{R}. In that sense, our result covers
 almost all interesting values of $p$. For the supercritical case, we
 only use the condition $p\ge n^{-1/8+\eps}$ for values of
 $\ell=n/2-O(n^{7/8+\eps})$ (see Theorem~\ref{t:nearperfect} and the
 remark below that). In the proof for other values of $\ell$, we only
 assume that $p=\omega(n^{-1/2})$. In fact, $p=\omega(n^{-1/2})$ is
 likely to be another necessary condition in the supercritical case
 since a result by Janson in~\cite{J3} implies that the hypotheses in
 the tool~\cite[Theorems 1 and 3]{G6} that we use will no longer be
 satisfied (for $\ell=n/2$). We conjecture that the condition $p\ge
 n^{-1/8+\eps}$ in our main theorem can be weakened to
 $p=\omega(n^{-1/2})$. The distribution of $X_{n,\ell}$ in the
 supercritical case for $p=O(n^{-1/2})$ remains open.

\section{Main results}

An $\ell$-matching is a matching with $\ell$ edges. Let $X =
X_{n,\ell}$ denote the number of subgraphs of $\Gnp$ that are
isomorphic to an $\ell$-matching. Throughout the paper, let $N =
\binom{n}{2}$ and define $m!!$ to be $\prod_{i=0}^{\lfloor
  (m-1)/2\rfloor} (m-2i)$ for any real number $m\ge 1$. Then, the
number of $\ell$-matchings in the complete graph $K_n$ is
\begin{equation}
\binom{n}{2\ell}(2\ell-1)!!=\binom{n}{2\ell}\frac{(2\ell)!}{2^{\ell}\ell!}. \lab{sizeOfs}
\end{equation}
 Let $\la_{n,\ell} := \ex{X_{n,\ell}}$ and $\sigma_{n,\ell} :=
\sqrt{\var{X_{n, \ell}}}$. Then, obviously,
\[
\la_{n,\ell}=\binom{n}{2\ell}\frac{(2\ell)!}{2^{\ell}\ell!}p^{\ell}.
\]
Define
\begin{equation}
  \label{eq:barsigma}
  \bar\sigma := \bar\sigma_{n,\ell} =
  \left(\ell\binom{n}{2\ell}\binom{n-2}{2\ell-2}
    (2\ell -1)!! (2\ell -3)!! (p^{2\ell-1}-p^{2\ell})\right)^{1/2}.
\end{equation}
We will show that $\sigma_{n,\ell}\sim \bar\sigma_{n,\ell}$. Moreover,
we will prove the following theorem about the central moments of
$X_{n,\ell}$.
\begin{thm}\lab{t:moments} Suppose that $1-p=\Omega(1)$. Then, for every positive $\ell=\ell(n)=o(n\sqrt{p})$ and for
  every fixed integer $k\ge 2$,
\begin{equation*}
\ex ((X_{n,\ell}-\la_{n,\ell})^k)=\left\{
\begin{array}{ll}
(1+o(1))(k-1)!!\bar\sigma^k, & \mbox{if $k$ is even};\\
o(\bar\sigma^k), & \mbox{if $k$ is odd}.
\end{array}
\right.
\end{equation*}
\end{thm}

By Theorem~\ref{t:moments} and using the method of moments
(Theorem~\ref{t:mm} below), we immediately have the following theorem
for the subcritical case.

\begin{thm}
  \lab{t:sub} Suppose that $1-p = \Omega(1)$. For every positive
  $\ell=\ell(n)=o(n\sqrt{p})$,
  \begin{equation*}
    \frac{X_{n,\ell}-\lambda_{n,\ell}}{\sigma_{n,\ell}}
    \xrightarrow{d}
    \N(0,1), \ \ \mbox{as}\ n\to\infty,
  \end{equation*}
  where $\N(0,1)$ is the standard normal distribution.
\end{thm}

\noindent{\bf Remark}: Note that condition $p=\omega(n^{-2})$ is implicit in Theorems~\ref{t:moments} and~\ref{t:sub} so that $\ell=o(n\sqrt{p})$ can be satisfied by some positive integer $\ell$.\ss  

The following result gives the asymptotic distribution of $X_{n,\ell}$ in the supercritical case.

\begin{thm}
  \lab{t:linear} Let $0<\eps<1/8$ be an arbitrarily small constant.
  Suppose that $1-p = \Omega(1)$ and $p \ge n^{-1/8+\eps}$. Then, for
  every positive $\ell=\ell(n)=\Omega(n\sqrt{p})$,
  $$
    \frac{\ln(e^{\beta_{n,\ell}^2/2} X_{n, \ell}/\la_{n,\ell})}
    {\beta_{n,\ell}}\xrightarrow{d}
    \N(0,1), \ \ \mbox{as}\ n\to\infty,
    $$
    where $\beta_{n,\ell}=\ell\sqrt{(1-p)/pN}$.
\end{thm}

Immediately, we have the following corollary of Theorems~\ref{t:sub} and~\ref{t:linear}.

\begin{cor}
  \lab{c:main} Let $0<\eps<1/8$ be fixed.  Suppose that $1-p
  = \Omega(1)$ and $p \ge n^{-1/8+\eps}$. For any integer $\ell
  = \ell(n)\in[1, n/2]$,
  \begin{itemize}
  \item[(i)] if $\ell = o(n\sqrt{p})$, then $X_{n,\ell}$ is asymptotically normally distributed;
  \item[(ii)] if $\ell = \Omega(n\sqrt{p})$, then
    $X_{n,\ell}$ is asymptotically log-normally distributed.
  \end{itemize}
\end{cor}

Theorem~\ref{t:linear} is implied by stitching together the following
two theorems, which give more
precise description of the distribution of $X_{n,\ell}$ (with wider
ranges of $p$ than in Theorem~\ref{t:linear}).

\begin{thm}
  \lab{t:super} Let $\alpha \in (7/8, 1)$ be fixed and suppose that
  $np\to \infty$ and $1-p=\Omega(1)$. Then, for every
  positive $\ell = \ell(n) =\Omega(n\sqrt{p})$
  such that $ \ell \le n/2 - n^{\alpha}$ and $ \ell^3 = o(n^4p^2)$, we
  have
  $$
  \frac{\ln(e^{\beta_{n,\ell}^2/2} X_{n,\ell}/\la_{n,\ell})}{\beta_{n,\ell}}\xrightarrow{d}
  \N(0,1) \ \ \mbox{as}\ n\to\infty,
  $$
  where $\beta_{n,\ell}=\ell\sqrt{(1-p)/pN}$ and $\N(0,1)$ is the
  standard normal distribution.
\end{thm}

\begin{thm}
  \lab{t:nearperfect}
  Let $\alpha \in(1/2,1)$ be fixed. Suppose 
  that $pn^{1-\alpha}\to\infty$ and $1-p=\Omega(1)$. Then, for every
  positive $\ell= \ell(n) = n/2-O(n^{\alpha})$,
  $$
  \frac{\ln(e^{\beta_{n,\ell}^2/2} X_{n,\ell}/\la_{n,\ell})}{\beta_{n,\ell}}\xrightarrow{d}
  \N(0,1) \ \ \mbox{as}\ n\to\infty,
  $$
  where $\beta_{n,\ell}=\ell\sqrt{(1-p)/pN}$ and $\N(0,1)$ is the
  standard normal distribution.
\end{thm}

\noindent {\bf Remark}: Theorems~\ref{t:super} and~\ref{t:nearperfect}
deal with the case $\ell=\Omega(n\sqrt{p})$. Note that the condition $\ell^3 =
o(n^4p^2)$ in Theorem~\ref{t:super} is weaker than the condition $p^2
n\to\infty$. Indeed, assuming $p^2n \to \infty$, we have that
  $\ell^3 \leq n^3 = o(n^3 \cdot p^2n)$. Thus, the condition $p\ge
n^{-1/8+\eps}$ in Theorem~\ref{t:linear} is only used while applying
Theorem~\ref{t:nearperfect} (by taking $\alpha=7/8+\eps$) and it is
likely that it may be weakened to $p^2 n\to\infty$, as
we conjectured in the introduction.

\section{Proof of Theorems~\ref{t:moments} and~\ref{t:sub}}
\lab{s:sub}

Theorem~\ref{t:sub} follows from Theorem~\ref{t:moments} and the
following theorem, known as the method of moments.
\begin{thm}[Corollary 6.3 in~\cite{JLR}]
  \lab{t:mm}
  If $Z_1,Z_2,\dotsc$ are random variables with finite moments and
  $a_n$ are positive numbers such that, for fixed integer $k\geq 2$,
  as $n\to \infty$,
  \begin{equation*}
    \mean{(Z_n - \ex Z_n)^k} =
    \begin{cases}
      (k-1)!! a_n^k + o(a_n^k), &\text{if }k\text{ is even};\\
      o(a_n^k),&\text{if }k\text{ is odd};
    \end{cases}
  \end{equation*}
  then $(Z_n-\ex Z_n)/\sqrt{\var{Z_n}} \xrightarrow{d}
  \N(0,1)$.
\end{thm}
We proceed to prove Theorem~\ref{t:moments}. Let $\Mcal =
\set{M_1,\dotsc, M_s}$ be the set of $\ell$-matchings in $K_{[n]}$,
where $K_{[n]}$ is the complete graph on $[n]$. Thus,
by~\eqn{sizeOfs}, $s=\binom{n}{2\ell}(2\ell-1)!!$. Let $\HM$ denote
the graph on $\Mcal$ such that a matching $M$ is adjacent to another
matching $M'$ if and only if $M\cap M' \neq\varnothing$. Let $k\ge 1$
be any fixed integer. For any $(i_1,\dotsc, i_k)\in[s]^k$, let
$\HM(i_1,\dotsc, i_k)$ be the subgraph of $\HM$ induced by
$\set{M_{i_1},\dotsc, M_{i_k}}$ and let $\Ccal(i_1,\dotsc, i_k)$
denote the set of components of $H_{\Mcal}(i_1,\dotsc, i_k)$. For
$C\in \Ccal(i_1,\dotsc, i_k)$, let ${\check n}_C = \card{\set{j:
    M_{i_j}\in V(C)}}$ and ${\check m}_C = \card{\bigcup_{M\in C} M}$.
That is, ${\check n}_C$ is the number of matchings in $C$ (counting
repetitions), whereas ${\check m}_C$ counts edges in the union of the
matchings in $C$. We will use $\iseq$ to denote $(i_1,\dotsc, i_k)$
and $C(\iseq)$ to denote $C(i_1,\dotsc, i_k)$.

Two matchings are called a {\em kissing pair} if they share exactly
one edge. An ordered triple of matchings $(M_1,M_2,M_3)$ is
called a {\em chained triple} if
$\card{M_{1}\cap M_{2}}=1$ and $\card{M_{2}\cap M_{3}} = 1$ and
$\card{M_{1}\cap M_{3}} = 0$. A set of matchings $\set{M_1,\dotsc,
  M_t}$ of size $t$ is called a {\em flower with $t$ petals} if there
exists an edge $e$ such that $M_i\cap M_j = \set{e}$ for any distinct
$i,j\in[k]$. Hence, a flower with two petals is a kissing pair.

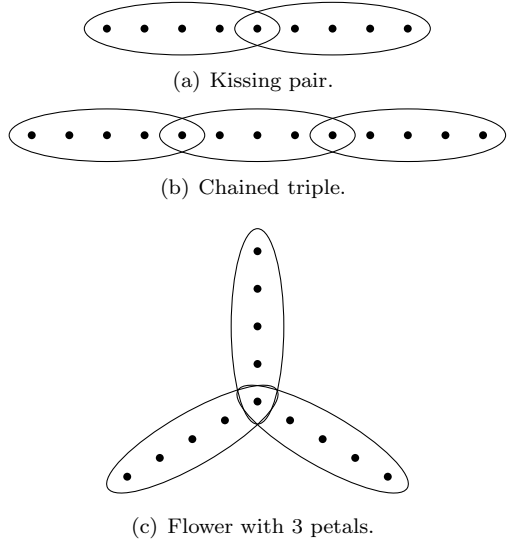
\begin{figure}
\centering
\subfigure[Kissing pair.]{
  \begin{tikzpicture}
    \def \initial {0};
    \def \step {0.5};
    \def \rad {1.5pt};
    \def \radv {3.5mm};
    \def \radh {13mm};

    \coordinate (a) at (\initial,0);
    \coordinate (b) at (\initial+\step,0);
    \coordinate (c) at (\initial+2*\step,0);
    \coordinate (d) at (\initial+3*\step,0);
    \coordinate (e) at (\initial+4*\step,0);
    \coordinate (f) at (\initial+5*\step,0);
    \coordinate (g) at (\initial+6*\step,0);
    \coordinate (h) at (\initial+7*\step,0);
    \coordinate (i) at (\initial+8*\step,0);

    \foreach \point in {a,b,c,d,e,f,g,h,i} \fill [black] (\point)
    circle (\rad);

    \draw (c) ellipse ({\radh} and {\radv});
    \draw (g) ellipse ({\radh} and {\radv});
  \end{tikzpicture}}
\hspace{1cm}
\subfigure[Chained triple.]{
  \begin{tikzpicture}
    \def \initial {0};
    \def \step {0.5};
    \def \rad {1.5pt};
    \def \radv {3.5mm};
    \def \radh {13mm};

    \coordinate (a) at (\initial,0);
    \coordinate (b) at (\initial+\step,0);
    \coordinate (c) at (\initial+2*\step,0);
    \coordinate (d) at (\initial+3*\step,0);
    \coordinate (e) at (\initial+4*\step,0);
    \coordinate (f) at (\initial+5*\step,0);
    \coordinate (g) at (\initial+6*\step,0);
    \coordinate (h) at (\initial+7*\step,0);
    \coordinate (i) at (\initial+8*\step,0);
    \coordinate (j) at (\initial+9*\step,0);
    \coordinate (k) at (\initial+10*\step,0);
    \coordinate (l) at (\initial+11*\step,0);
    \coordinate (m) at (\initial+12*\step,0);

    \foreach \point in {a,b,c,d,e,f,g,h,i,j,k,l,m} \fill [black] (\point)
    circle (\rad);

    \draw (c) ellipse ({\radh} and {\radv});
    \draw (g) ellipse ({\radh} and {\radv});
    \draw (k) ellipse ({\radh} and {\radv});
  \end{tikzpicture}}
\\
\subfigure[Flower with $3$ petals.]{
  \begin{tikzpicture}
    \def \initial {0};
    \def \step {0.5};
    \def \rad {1.5pt};
    \def \radv {3.5mm};
    \def \radh {13mm};

    \coordinate (a) at (0,\initial);
    \coordinate (b) at (0,\initial+\step);
    \coordinate (c) at (0,\initial+2*\step);
    \coordinate (d) at (0,\initial+3*\step);
    \coordinate (e) at (0,\initial+4*\step);

    \coordinate (f) at ($ (a) + \step*({cos(210)},{sin(210)}) $) ;
    \coordinate (g) at ($ (a) + 2*\step*({cos(210)},{sin(210)}) $) ;
    \coordinate (h) at ($ (a) + 3*\step*({cos(210)},{sin(210)}) $) ;
    \coordinate (i) at ($ (a) + 4*\step*({cos(210)},{sin(210)}) $) ;

    \coordinate (j) at ($ (a) + \step*({cos(330)},{sin(330)}) $) ;
    \coordinate (k) at ($ (a) + 2*\step*({cos(330)},{sin(330)}) $) ;
    \coordinate (l) at ($ (a) + 3*\step*({cos(330)},{sin(330)}) $) ;
    \coordinate (m) at ($ (a) + 4*\step*({cos(330)},{sin(330)}) $) ;

    \foreach \point in {a,b,c,d,e,f,g,h,i,j,k,l,m} \fill [black] (\point)
    circle (\rad);

    \draw (c) ellipse ({\radv} and {\radh});
    \draw[rotate=210] (g) ellipse ({\radh} and {\radv});
    \draw[rotate=330] (k) ellipse ({\radh} and {\radv});
  \end{tikzpicture}}
\caption{Kissing pair, chained triple and flower with $3$ petals. Each dot represents
  an edge and each ellipse represents a matching.}
\label{fig:structures}
\end{figure}

Let $\Kcal$ denote the subset of $[s]^k$ such that $\iseq\in \Kcal$ if
each component $C \in \Ccal(\iseq)$ satisfies $\cn_C \geq 2$.  Let
$\Kcal'$ be the subset of $\Kcal$ such that $\iseq\in\Kcal'$ if
\begin{description}
\item{(a)} $\card{\Ccal(\iseq)}=\floor{k/2}$ (and so
  $\cn_C\in\set{2,3}$ for every $C\in \Ccal(\iseq)$);
\item{(b)} if $\cn_C = 2$ for some $C \in \Ccal(\iseq)$, then $C$ is
   a kissing pair;
\item{(c)} if $\cn_C = 3$ for some $C \in \Ccal(\iseq)$, then $C$ is
   a chained triple or it is a flower with $3$ petals.
\end{description}

\noindent {\bf Remark}: By the definition of $\Kcal'$, if $k$ is even,
then every component of $\Ccal({\bf i})$ for ${\bf i}\in \Kcal'$ is a
kissing pair; whereas if $k$ is odd, every component in $\Ccal({\bf
  i})$ is a kissing pair except one, which is a chained triple or a
flower with $3$ petals.\smallskip

For $i\in[s]$, let $X_i$ be the indicator variable for the event that
$M_i\subseteq \Gnp$ and let $Y_i = X_i - \ex X_i$. Note that $\ex
X_i=p^{\ell}$ for all $i\in [s]$. We first estimate $\ex
\big(\prod_{j=1}^k Y_{i_j}\big)$, for ${\bf i}\in \Kcal'$.

\begin{prop} \lab{p:prodY}
For any $\iseq=(i_1,\ldots,i_k)\in \Kcal'$,
if $k$ is even,
\begin{eqnarray*}
  \ex\left(\prod_{j=1}^k Y_{i_j}\right)=(p^{2\ell-1}-p^{2\ell})^{k/2};
\end{eqnarray*}
 if $k$ is odd,
 \begin{eqnarray*}
 \left| \ex\left(\prod_{j=1}^k Y_{i_j}\right)\right|
  \le (p^{2\ell-1}-p^{2\ell})^{(k-3)/2}p^{3\ell-2}.
\end{eqnarray*}
\end{prop}
\begin{proof}
  We only give a detailed proof for the case that $k$ is even and the proof for
  the other case
 is analogous. By the definition of $\Kcal'$, $\Ccal(\iseq)$
contains $k/2$ components, each of which is a kissing pair. Without
loss of generality, we may assume that $\{M_{i_{2j-1}},M_{i_{2j}}\}$,
$1\le j\le k/2$ are kissing pairs. Since each kissing pair does not
share any edge with other kissing pairs, $Y_{i_{2j-1}}$ is independent
with all other $Y_{i'}$, $i'\in \iseq$, except for
$Y_{i_{2j}}$. Hence,
$$
\ex\left(\prod_{j=1}^k Y_{i_j}\right)=\prod_{j=1}^{k/2}\ex\left(
  Y_{i_{2j-1}}Y_{i_{2j}}\right)=\ex\left(
  Y_{i_{1}}Y_{i_{2}}\right)^{k/2}.
$$
Since $M_{i_1}$ and $M_{i_2}$ share exactly one edge, we have $|M_1\cap
M_2|=2\ell-1$ and so
$$
\ex \paren{Y_{i_{1}}Y_{i_{2}}}= \ex \paren{X_{i_{1}}X_{i_{2}}}-p^{2\ell}=p^{2\ell-1}-p^{2\ell}.
$$
This completes the proof for the case that $k$ is even. For odd $k$, we can easily show that,
if the component with three matchings is a chained
triple, then
\begin{eqnarray*}
  \ex\left(\prod_{j=1}^k Y_{i_j}\right)
  =(p^{2\ell-1}-p^{2\ell})^{(k-3)/2}p^{3\ell-2}(1-p)^2,
\end{eqnarray*}
and, if the component with three matchings is a
flower with $3$ petals, then
\begin{eqnarray*}
  \ex\left(\prod_{j=1}^k Y_{i_j}\right)
  =(p^{2\ell-1}-p^{2\ell})^{(k-3)/2}p^{3\ell-2}(1-3p+2p^2).
\end{eqnarray*}
Since $0\le(1-p)^2\le 1$ and $|1-3p+2p^2|\le 1$, the inequality in the lemma follows.
\end{proof}

The following lemma shows that the
leading contribution to the $k$-th central moment of $X_{n,\ell}$ is
from graph structures in $\Kcal'$.
\begin{lemma}
  \lab{l:sub-significant} Suppose that $1-p = \Omega(1)$ and $\ell^2 =
  o(n^2p)$. For every fixed even $k\in\setN$ with $k\geq 2$,
  \begin{equation}
    \label{eq:small_even}
    \mean[\big]{(X-\ex X)^k}
    =
    \sum_{\iseq \in \Kcal}
    \mean[\Big]{\prod_{j=1}^k Y_{i_j}}
    \sim
    \sum_{\iseq \in \Kcal'}
    \mean[\Big]{\prod_{j=1}^k Y_{i_j}}
    =
    \card{\Kcal'}
    (p^{2\ell-1}-p^{2\ell})^{k/2};
  \end{equation}
  and, for every odd $k\in\setN$ with $k\ge 3$,
  \begin{equation}
    \label{eq:small_odd}
    \abs[\Big]{\mean[\big]{(X-\ex X)^k}}
    \leq
    \sum_{\iseq \in \Kcal} \abs[\Big]{
    \mean[\Big]{\prod_{j=1}^k Y_{i_j}}}
    \le
    (1+o(1)) \card{\Kcal'} (p^{2\ell-1}-p^{2\ell})^{(k-3)/2}p^{3\ell-2}.
  \end{equation}
\end{lemma}

\noindent {\bf Remark}: Note that
$(p^{2\ell-1}-p^{2\ell})^{(k-3)/2}p^{3\ell-2}$ in~\eqref{eq:small_odd} is
an upper bound of $\abs[\Big]{ \mean[\Big]{\prod_{j=1}^k Y_{i_j}}}$
for any ${\bf i}\in \Kcal'$ by Proposition~\ref{p:prodY}. In fact, it
is possible to strengthen \eqref{eq:small_odd} to $
\abs[\big]{\mean[\big]{(X-\ex X)^k}} \leq (1+o(1)) \sum_{\iseq \in
  \Kcal'} \abs[\big]{\mean[\big]{\prod_{j=1}^k Y_{i_j}}}$ with a
slightly more delicate proof. But, an upper bound as
in~\eqref{eq:small_odd} is sufficient and it allows us to present a
slightly simpler proof.

We leave Lemma~\ref{l:sub-significant} to be proved in
Section~\ref{s:sub-significant}. Now we complete the proof of
  Theorem~\ref{t:moments} by assuming Lemma~\ref{l:sub-significant}.

\begin{proof}[Proof of Theorem~\ref{t:moments}]
  Suppose that $k$ is even.  By Lemma~\ref{l:sub-significant}, it
    suffices to show that
    $\card{\Kcal'}(p^{2\ell-1}-p^{2\ell})^{k/2}\sim
    (k-1)!!\bar\sigma^k$. Recall that $\iseq\in \Kcal'$ if
  $\card{\Ccal(\iseq)}= k/2$ and each $C\in\Ccal(\iseq)$ is a kissing
  pair. First we
    evaluate the size of the set $\Tcal$ of $(k/2)$-tuples $((i_1',
    i_2'),\dotsc, (i_{k-1}', i_{k}'))\in ([s]\times[s])^{k/2}$ such
    that, for each pair $(i_{j}', i_{j+1}')$ with odd $j\in[k-1]$, we
    have that $M_{i_{j'}}$ and $M_{i_{j+1}'}$ are a kissing pair and
    $|M_{i_j'}\cap M_{i_t'}| = 0$ and $|M_{i_{j+1}'}\cap M_{i_t'}| =
    0$ for all $t < j$. Given any $\iseq' \in \Tcal$ and any perfect
    matching $P$ of $[k]$, we obtain a $k$-tuple $\iseq\in\Kcal'$ as
    follows.  Order the edges in $P$ as $(u_1v_1,\dotsc,
    u_{k/2}v_{k/2})$ in a way such that, for every $j\in[k/2]$, $u_j<
    v_j$ and $u_j < u_{j'}$ for any $j' > j$. Then set $i_{u_j} =
    i_{2j-1}'$ and $i_{v_j} = i_{2j}'$ for all $1\leq j\leq k/2$.
    Each $\iseq$ in $\Kcal'$ is generated by a unique pair $(\iseq',
    P)$, where $\iseq'\in\Tcal$ and $P$ is a perfect matching on
    $[k]$.  Since there are exactly $(k-1)!!$ such matchings, we have
    that
  \begin{equation}
  |\Kcal'|=(k-1)!!|\Tcal|. \lab{sizeOfK'}
\end{equation}
Note that $|V(H_\Mcal)| = s$. Recall from~\eqn{sizeOfs} that
$s=\binom{n}{2l}(2l-1)!!$.  Note also that $H_{{\cal M}}$ is a regular
graph. Let $D$ denote the degree of any vertex in $H_{\Mcal}$ and let
$d$ denote the number of $\ell$-matchings with exactly one edge in
common with a given fixed $\ell$-matching.  Let $\Delta_r$ denote the
number of $\ell$-matchings containing a given $r$-matching. We have
that
  \begin{equation}
    \label{eq:nD}
    D\leq \ell \Delta_1 = \ell \binom{n-2}{2l-2}(2l-3)!!
    = O\paren[\Big]{s \frac{\ell^2}{n^2}} = o(s)
  \end{equation}
  since $\ell^2/n^2 = o(1)$. Moreover, by the
  Inclusion-Exclusion Principle,
  \begin{equation}
    \label{eq:d-sub}
    \begin{split}
      &d \leq \ell \Delta_1
    \quad\text{and}\\
    &d\geq \ell \Delta_1 - \binom{\ell}{2}\Delta_2
    =
    \ell \binom{n-2}{2\ell-2}(2l-3)!!-
    \binom{\ell}{2}\binom{n-4}{2\ell-4}(2\ell-5)!!
    \sim \ell\Delta_1
    \end{split}
  \end{equation}
  since $\binom{\ell}{2}\Delta_2/(\ell\Delta_1) = O(\ell^2/n^2) =
  o(1)$. Thus, we have shown that $d\sim \ell\Delta_1$. Suppose we
  already chose $(i_1', i_2'), \dotsc, (i_{j-2}', i_{j-1}')$. We
  compute the number of choices for $(i_{j}', i_{j+1}')$. We have at
  most $s$ choices for $i_{j}'$ and, using~\eqref{eq:nD}, at least $s
  - kD\sim s$ choices. Next, we estimate the number of choices for
  $i_{j+1}'$. The matching $M_{i_{j+1}'}$ must be chosen among the
  ones that have exactly one edge in common with $M_{i_{j}'}$. Hence,
  the number of choices is at most $d$. On the other hand, a matching
  containing exactly one edge $e_1$ in common with $M_{i_j'}$ cannot
  be chosen as $M_{i_{j+1}'}$ only if it has
  another edge $e_2$ in common with some matching $M_{i_{t}'}$ with $t
  < j$. There are $\ell$ choices for $e_1$ and at most $k\ell$ choices
  for $e_2$. The number of $\ell$-matchings containing $e_1$ and $e_2$
  is at most $\Delta_2$. Thus, the number of choices for
  $M_{i_{j+1}'}$ is at least $d-k\ell^2\Delta_2$. By~\eqref{eq:d-sub},
  $d-k \ell^2\Delta_2 \sim d$ and so the number of choices for
  $i_{j+1}'$ is asymptotically $d$. Thus, $|\Tcal| \sim (s d)^{k/2}$,
  and we are done by~\eqref{eq:nD},~\eqref{eq:d-sub}
  and~\eqn{sizeOfK'}, finishing the proof for even $k$.

    Now suppose that $k$ is odd. By
      Lemma~\ref{l:sub-significant},~\eqref{eq:barsigma},
      and~\eqref{eq:d-sub}, it suffices to show that
    \begin{equation}
      \label{eq:sizeK'goalodd}
      \card{\Kcal'}(p^{2\ell-1}-p^{2\ell})^{(k-3)/2}p^{3\ell-2}
      = o
      \left(\left(sd (p^{2\ell-1}-p^{2\ell})\right)^{k/2}\right).
    \end{equation}
 Recall that $\iseq\in \Kcal'$ if
  $\Ccal(\iseq)$ contains $(k-1)/2$ components, in which $(k-3)/2$ are
  kissing pairs and the other one is a chained triple or a flower with
  $3$ petals. Similarly to the previous case when $k$ is even, we
  consider the set $\Tcal$ of $\floor{k/2}$-tuples
  $((i_1', i_2'),\dotsc, (i_{k-2}', i_{k-3}'), (i_{k-2}', i_{k-1}',
  i_k'))$ with every $i_j'\in [s]$, such that all $\{M_{i_{2j-1}'},
  M_{i_{2j}'}\}$, $1\le j\le (k-3)/2$, are kissing pairs and
  $(M_{i_{k-2}'},M_{i_{k-1}'}, M_{i_k'})$ is a chained triple or a
  flower with $3$ petals and all kissing pairs and chained triples (or
  the flower) are edge disjoint. Similar to the previous argument, we
  have $|\Kcal'|=O(|\Tcal|)$ since $k$ is fixed. The number of
  choices for the first $(k-3)/2$ pairs in $T\in\Tcal$ is at most $(s
  d)^{(k-3)/2}$ and the number of choices for the triple
  $(i_{k-2}',i_{k-1}', i_k')$ is at most $2sd^2$ (at most $sd^2$
  choices for a chained triple and at most $sd^2$ choices for a flower
  with $3$ petals). Thus, we have that $|\Kcal'| = O\left( (s d)^{k/2} \cdot
    d/\sqrt{sd} \right)$ and, in order to
  prove~\eqref{eq:sizeK'goalodd}, it suffices to show that $d
  p^{3\ell-2} = o(\sqrt{sd(p^{2\ell-1}-p^{2\ell})^3})$. Indeed,
  using~\eqref{eq:nD} and~\eqref{eq:d-sub},
  \begin{equation*}
    \frac{d^2 p^{6\ell-4}}{(p^{2\ell-1}-p^{2\ell})^3sd}
    = O\left(\frac{d p^{6\ell-4}}{s p^{6\ell-3}(1-p)^3} \right)
    = O\left(\frac{\ell^2}{n^2p (1-p)^3}\right),
  \end{equation*}
  which goes to zero since $\ell^2/(n^2p) = o(1)$ and $1-p =
  \Omega(1)$.
\end{proof}

\subsection{Proof of Lemma~\ref{l:sub-significant}}
\lab{s:sub-significant}

In this section, we assume $\ell = o(n\sqrt{p})$ and $1-p = \Omega(1)$
and $k\in \setN$ is fixed, which are the hypotheses of
Lemma~\ref{l:sub-significant}. Recall that for $C\in\Ccal(\iseq)$ we
defined $\cn_C = \card{\set{j: M_{i_j}\in V(C)}}$ and $\cm_C =
\card{\bigcup_{M\in C} M}$. For each $C\in\Ccal(\iseq)$ where
$\iseq\in [s]^k$, define
$$
Y_C=\prod_{j: M_{i_j}\in C} Y_{i_j}.
$$
and
define
\begin{equation*}
  {\check p}_C =
  \begin{cases}
    p^{\cm_C}-p^{{\check n}_C\ell},& \text{if }{\check n}_C\le 2;\\
    p^{\cm_C},& \text{otherwise};
  \end{cases}
\end{equation*}
define then
\begin{equation*}
  \cp(\iseq)
  =
  \prod_{C\in\Ccal(\iseq)} \cp_{C}.
\end{equation*}

\begin{lemma}
  \label{lem:uppercp}
  For every $\iseq\in [s]^k$ and any $C\in\Ccal(\iseq)$,
  \begin{equation*}
    \left|
      \ex Y_C
    \right|
    \leq
    2^{\cn_C}{\check p}_C.
  \end{equation*}
  Moreover, if $\cn_C \le 2$, then $ \ex Y_C  =\cp_C$.
  In particular, $\cp_C=0$ if $\cn_C=1$.
\end{lemma}
\begin{proof}
  This follows from the fact that, for any subset $I\in[k]$, we have
  that
  \begin{equation*}
    \begin{split}
      \left|
      \mean[\Big]{\textstyle \prod_{i\in I} Y_{i}}
    \right|
    &=\left|
      \mean[\Big]{\textstyle \prod_{i\in I} (X_{i}-p^{\ell})}
    \right|
    =
    \left|\sum_{I'\subseteq I} (-1)^{|I\setminus I'|}
      p^{\ell|I\setminus I'|}p^{\card{\cup_{i\in I'}M_i}}
    \right|
    \\
    &\leq
    2^{|I|} p^{\card{\cup_{i\in I}M_i}}
    \end{split}
  \end{equation*}
  since there are $2^{|I|}$ choices for $I'$ and, for each $I'$, we have
  \[
  \card{\cup_{i\in I}M_i}\le \card{\cup_{i\in I'}M_i}+\card{\cup_{i\in I\setminus I'}M_i}\le \card{\cup_{i\in I'}M_i}+\ell|I\setminus I'|.
  \]
   The conclusion for the case where $\cn_C \le 2$ follows
  directly from the definition of~$\cp_C$.
\end{proof}

\begin{proof}[Proof of Lemma~\ref{l:sub-significant}]
  For any $\iseq\in[s]^k$,
  \begin{equation}
  \mean[\Big]{\prod_{j=1}^k Y_{i_j}}=\prod_{C\in \Ccal(\iseq)} \ex Y_C,\lab{eq:YC}
  \end{equation}
  since $Y_C$, $C\in \Ccal(\iseq)$ are independent random variables.
  If $\Ccal(\iseq)$ has a component $C$ with $\cn_C = 1$, then it
  follows directly from~\eqn{eq:YC} and Lemma~\ref{lem:uppercp} that
  $\mean[\Big]{\prod_{j=1}^k Y_{i_j}}=0$. Thus, we have
  $\mean{(X-\mean{X})^k} = \sum_{\iseq \in \Kcal}
  \mean[\Big]{\prod_{j=1}^k Y_{i_j}}$. 
  By Proposition~\ref{p:prodY} and the definition of $\cp_C$,
    we have that
  \begin{equation}
    \mean[\Big]{\prod_{j=1}^k Y_{i_j}}=\cp(\iseq)\ \mbox{if $k$ is even}; \quad
    \left|\mean[\Big]{\prod_{j=1}^k Y_{i_j}}\right| \leq \cp(\iseq)\
  \mbox{if $k$ is odd}. \lab{eq:expYC}
  \end{equation}
  Next, we prove that
  \begin{equation}
    \lab{eq:significant}
    \sum_{\iseq \in \Kcal}
    \cp(\iseq)
    \sim
    \sum_{\iseq \in \Kcal'}
    \cp(\iseq).
  \end{equation}
  Note that Lemma~\ref{l:sub-significant} then follows from the above
  and the definition of $\cp(\iseq)$. Indeed, for even $k$,
  we have that
  \begin{equation*}
    \begin{split}
      \sum_{\iseq \in \Kcal}
      \mean[\Big]{\prod_{j=1}^k Y_{i_j}}
      &=
      \sum_{\iseq \in \Kcal'}
      \mean[\Big]{\prod_{j=1}^k Y_{i_j}}
      +
      \sum_{\iseq \in \Kcal\setminus \Kcal'}
      \mean[\Big]{\prod_{j=1}^k Y_{i_j}}
      \\
      &=
      \sum_{\iseq \in \Kcal'}
      \cp(\iseq)
      +
      O\paren[\Big]{
        \sum_{\iseq \in \Kcal\setminus \Kcal'}
        \cp(\iseq)
      }
      \sim
      \sum_{\iseq \in \Kcal'}
      \cp(\iseq),
    \end{split}
  \end{equation*}
  where the second equality follows by~\eqref{eq:expYC} and
  Lemma~\ref{lem:uppercp}, and the last relation holds
  by~\eqref{eq:significant}. Similarly, for odd $k$,
  \begin{equation*}
    \begin{split}
      \left|\sum_{\iseq \in \Kcal}
      \mean[\Big]{\prod_{j=1}^k Y_{i_j}}\right|
      &\leq
      \sum_{\iseq \in \Kcal'}
      \left|\mean[\Big]{\prod_{j=1}^k Y_{i_j}}\right|
      +
      \sum_{\iseq \in \Kcal\setminus \Kcal'}
      \left|\mean[\Big]{\prod_{j=1}^k Y_{i_j}}\right|
      \\
      &\leq
      \sum_{\iseq \in \Kcal'}
      \cp(\iseq)
      +O\left(
      \sum_{\iseq \in \Kcal\setminus \Kcal'}
        \cp(\iseq)
      \right)
      \sim
      \sum_{\iseq \in \Kcal'}
      \cp(\iseq).
    \end{split}
  \end{equation*}

Thus, it suffices to show~\eqref{eq:significant}. Let
$\Kcal(x_1,\dotsc, x_k)$ be the restriction of $\Kcal\setminus \Kcal'$
to the tuples $(i_1,\dots, i_k)$ such that, for every $1\leq j\leq k$,
the number of edges $e$ with $\card{\set{j': e\in M_{i_{j'}}}} = j$ is
$x_j$, that is, the number of edges that appear in exactly $j$
matchings is $x_j$. For any such tuple
$(x_1,\ldots,x_k)$, we have $\sum_{i=1}^k
ix_i=k\ell$. Since $\iseq\in \Kcal$, the number of edges contained
only in $M_{i_j}$ is at most $\ell-1$ for every $1\le j\le k$. It
follows immediately that $x_1\le k(\ell-1)$ and so we must have
$\sum_{i\geq 2}ix_i\geq k$.  Let $\Xcal = \set{(x_1,\dotsc,
    x_k)\in\setN^k:\sum_{i\geq 1}i x_i = k\ell,\ \sum_{i\geq 2}i x_i
    \geq k}$.  In order to prove~\eqref{eq:significant}, we will
define switchings from $\Kcal(x_1,\dotsc, x_k)$ to $\Kcal'$ for every
$(x_1,\dotsc,x_k)\in \Xcal$ and thereby we prove that the contribution
to~\eqref{eq:significant} from $\Kcal$ is dominated by the
contribution from $\Kcal'$. We discuss the cases when $k$ is odd
and even separately.

  We first prove~\eqref{eq:significant} for even $k$.  Let
  $(x_1,\dotsc, x_k)\in \Xcal$. Let ${\bf i}=(i_1,\dotsc, i_k)\in
  \Kcal(x_1,\dotsc, x_k)$. We define the following switching from
  ${\bf i}$ to $k$-tuples in $\Kcal'$ (see
  Figure~\ref{fig:switching-sub}).  For each $j$, let $I_j$ denote the
  set of edges that $M_{i_j}$  shares with $\cup_{j'\neq j}M_{i_{j'}}$.
  \begin{enumerate}
  \item ({\em Delete the shared edges}) For each $j\in[k]$, let $M_{i_j}' := M_{i_j}\setminus I_j$.

  \item ({\em Obtain pairwise disjoint partial matchings each of size $\ell-1$}) For each $j\in[k]$, choose edges $a_1^{(j)},\dotsc,
    a_{|I_j|-1}^{(j)}$, one after the other, so that
    $M_{i_j}'\cup\set{a_1^{(j)},\dotsc,a_{|I_{j}|-1}^{(j)}}$ is an
    $(\ell-1)$-matching and,  for every
    $r$, the edge $a_r^{(j)}$ is not in $\left(\bigcup_{j'=1}^{k}
      M_{i_{j'}}'\right) \cup \left(\bigcup_{j'<j}
      \set{a_1^{(j')},\dotsc,a_{|I_{j'}|-1}^{(j')}}\right)$.  Let $M_{i_j}'' \eqdef M_{i_j}' \cup
    \set{a_1^{(j)},\dotsc,a_{|I_{j}|-1}^{(j)}}$.
  \item ({\em Build kissing pairs}) Choose a perfect matching~$P$
         in~$[k]$. (Here we are choosing the pairs of matchings that
         will form a kissing pair in $\Kcal'$.)
       \item ({\em Choose shared edges in kissing pairs}) Let
         $e_1,\dotsc, e_{k/2}$ be an enumeration of the edges
         in~$P$. For each $r\in [k/2]$, let $j$ and $j'$ denote the
         ends of~$e_r$ and choose an edge $f_r$ such that $f_r\not\in
         \left(\bigcup_{a=1}^k M_{i_a}''\right)\cup
         \left(\bigcup_{b=1}^{r-1}f_b\right)$ and both $M_{i_j}''' :=
         M_{i_j}''\cup\set{f_r}$ and $M_{i_{j'}}''' :=
         M_{i_{j'}}''\cup\set{f_r}$ are matchings. (Note that
         $M_{i_j}'''$ and $M_{i_{j'}}'''$ are each of size $\ell$ and
         they form a kissing pair.)
       \item ({\em Update the indices}) Let $\iseq'''$ be the new
         tuple such that for every $1\le j\le k$,
         $M_{i'''_j}=M'''_{i_j}$.
   \end{enumerate}
   Let
   \begin{equation}
   L(x_1,\dotsc, x_k)={\left(\frac{1}{2} \binom{n}{2}\right)^{k\ell-x_1-k/2}}.\lab{eq:L}
   \end{equation}
   Now we show that, for every $(x_1,\ldots,x_k)\in \Xcal$ and for
   every ${\bf i}\in \Kcal(x_1,\ldots,x_k)$, the number of applicable
   switchings defined above for ${\bf i}$ is at least
   $L(x_1,\ldots,x_k)$, regardless of the choice of ${\bf i}$, for all
   sufficiently large $n$. In Step~2,
   given $r\in[|I_j|-1]$, we have that any edge with no ends in the set of vertices
   induced by the set of edges $(\bigcup_{j'=1}^{k} M_{i_{j'}}') \cup
   (\bigcup_{j'<j} \set{a_{1}^{(j')},\dotsc, a_{|I_{j'}|-1}^{(j')}})
   \cup(\bigcup_{r'<r}a_{r'}^{(j)})$ is a possible choice for
   $a_r^{(j)}$. Since this set has at most $k\ell$ edges, it induces
   at most $2k\ell$ vertices and so there are at least
   $\binom{n}{2}-2k\ell n$ choices for $a_r^{(j)}$ in Step~2 and thus
   we have at least $\left(\binom{n}{2}-2k\ell n
   \right)^{|I_j|-1}$ choices in Step~2.  There are
   $k!/((k/2)!2^{k/2})\geq 1$ choices for $P$ in Step~3. With
     the same argument as before, given the choice of $P$ in Step~3,
   there are at least $\left(\binom{n}{2}-2k\ell n \right)$ choices
   for $f$ for each pair in $P$ in Step~4. Since there are $k/2$
     pairs in $P$ in total, it follows then that the number of
   applicable switchings is at least
   \begin{equation}
     \begin{split}
      {\left(\binom{n}{2}-2k\ell n \right)^{k/2+\sum_{j=1}^k
         (|I_j|-1)}}
       &=
       {\left(\binom{n}{2}-2k\ell n \right)^{-k/2+\sum_{j=1}^k
         |I_j|}}\\
     &\ge L(x_1,\dotsc, x_k),
     \end{split}\lab{eq:switch-even}
   \end{equation}
   for all sufficiently large $n$,
where we used the fact that $\ell^2 =
   o(n^2p)$ and that
   $\sum_{j}\card{I_j} = k\ell-x_1$ in
   the last inequality.

   Now we describe the inverse switching, which converts ${\bf i}=(i_1,\dotsc,
   i_k)\in\Kcal'$ to some $k$-tuples in $\Kcal(x_1,\ldots,x_k)$.
   \begin{enumerate}
   \item \textit{(Choose the number of edges shared by each matching)}
     Choose an integral vector $\rseq = (r_1,\dotsc, r_k)$ so that
     $\sum_{j=1}^{k} r_j = k\ell-x_1$ such that $r_j\geq 1$ for every
     $j$. (In the following steps, we will convert ${\bf i}$ to some
     ${\bf i}''$ so that $M_{i''_j}$ contains
     $r_j$ shared edges)
   \item \textit{(Make room for shared edges)} For each
     $j\in[k]$, let $f_j$ be the unique edge in $M_{i_j}\cap
     (\bigcup_{j'\neq j} M_{i_{j'}})$.  Choose $r_j-1$ edges one by
     one without repetition in $M_{i_j}-f_j$. Let $M_{i_j}'$ be the
     $(\ell-r_j)$-matching obtained from $M_{i_j}$ by deleting these
     edges and $f_j$.
   \item \textit{(Choose shared edges)} Choose a set of edges $X$ of
     size $\sum_{i=2}^k x_i$ such that no edge in $X$ is contained
   in $\bigcup_{j=1}^k M_{i_j}'$.
   Partition $X$ into
     sets $X_2,\dotsc, X_k$ such that $|X_i|= x_i$ for every $i$.
   \item \textit{(Assign shared edges)} Construct a bipartite graph
     $Q$ with bipartition $(X,[k])$ so that the degree of each $e\in
     X_i$ is $i$ and the degree of each $j\in[k]$ is $r_j$.  Let
     $M_{i_j}''$ be obtained from $M_{i_j}'$ by including the edges in
     $X$ that are adjacent to $j$ in~$Q$.
   \item \textit{(Update the indices)} If each $M_{i_{j}}''$ is an
     $\ell$-matching, let $\iseq''$ be the new tuple such that for
     every $1\le j\le k$, $M_{i''_j}=M''_{i_j}$.
   \end{enumerate}

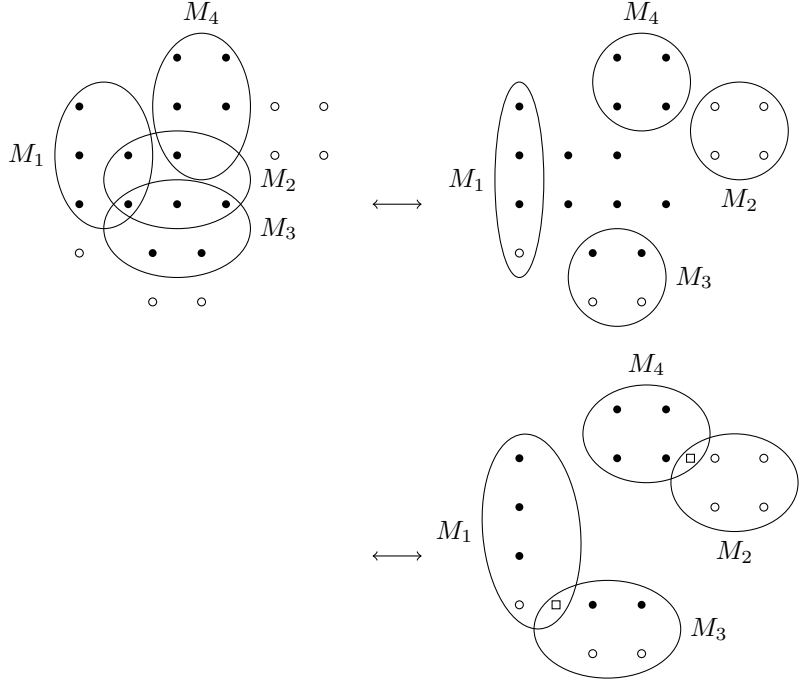
\begin{figure}
\centering
  \begin{tikzpicture}
    \def \initial {0};
    \def \step {0.65};
    \def \rad {1.5pt};

    \coordinate  (1) at (\initial,\initial);
    \coordinate  (2) at (\initial,\initial-1*\step);
    \coordinate  (3) at (\initial,\initial-2*\step);
    \coordinate  (4) at (\initial,\initial-3*\step);
    \coordinate  (5) at (\initial+\step,\initial-\step);
    \coordinate  (6) at (\initial+\step,\initial-2*\step);
    \coordinate  (7) at ($(5)+(\step, 0)$);
    \coordinate  (8) at ($(6)+(\step, 0)$);
    \coordinate  (9) at ($(8)+(\step, 0)$);
    \coordinate  (10) at ($(6)+(.5*\step, -1*\step)$);
    \coordinate  (11) at ($(10)+(\step, 0)$);
    \coordinate  (12) at ($(10)+(0, -\step)$);
    \coordinate  (13) at ($(12)+(\step, 0)$);
    \coordinate  (14) at ($(7)+(0,\step)$);
    \coordinate  (15) at ($(14)+(0,\step)$);
    \coordinate  (16) at ($(14)+(\step,0)$);
    \coordinate  (17) at ($(16)+(0,\step)$);
    \coordinate  (18) at ($(7)+(2*\step,0)$);
    \coordinate  (19) at ($(18)+(\step,0)$);
    \coordinate  (20) at ($(18)+(0,\step)$);
    \coordinate  (21) at ($(20)+(\step,0)$);

    \foreach \point in {1,2,3,5,6,7,8,9,10,11,14,15,16,17}
    \fill [black] (\point) circle ({\rad});

    \foreach \point in {4,12,13,18,19,20,21}
    \filldraw[fill=white]  (\point) circle ({\rad});

    \coordinate  (c1) at ($.5*(2)+.5*(5)$);
    \coordinate  (c2) at ($.5*(7)+.5*(8)$);
    \coordinate  (c3) at ($1/2*(8)+1/4*(10)+1/4*(11)$);
    \coordinate  (c4) at ($.5*(14)+.5*(16)$);

    \draw (c1) ellipse ({\step} and {1.5*\step});
    \coordinate [label=left:$M_1$] (m1) at ($(c1)-(\step,0)$);

    \draw (c2) ellipse ({1.5*\step} and {\step});
    \coordinate [label=right:$M_2$] (m2) at ($(c2)+(1.5*\step,0)$);

    \draw (c3) ellipse ({1.5*\step} and {\step});
    \coordinate [label=right:$M_3$] (m3) at ($(c3)+(1.5*\step,0)$);

    \draw (c4) ellipse ({\step} and {1.5*\step});
    \coordinate [label=above:$M_4$] (m4) at ($(c4)+(0,1.5*\step)$);

    \coordinate  (startarrow) at ($(19)+(\step,-\step)$);
    \coordinate  (endarrow) at ($(startarrow)+(\step,0)$);
    \draw [<->] (startarrow)--(endarrow);

    \def \drift {9*\step};
    \coordinate  (1a) at ($(1)+(\drift,0)$);
    \coordinate  (2a) at ($(2)+(\drift,0)$);
    \coordinate  (3a) at ($(3)+(\drift,0)$);
    \coordinate  (4a) at ($(4)+(\drift,0)$);
    \coordinate  (5a) at ($(5)+(\drift,0)$);
    \coordinate  (6a) at ($(6)+(\drift,0)$);
    \coordinate  (7a) at ($(7)+(\drift,0)$);
    \coordinate  (8a) at ($(8)+(\drift,0)$);
    \coordinate  (9a) at ($(9)+(\drift,0)$);
    \coordinate  (10a) at ($(10)+(\drift,0)$);
    \coordinate  (11a) at ($(11)+(\drift,0)$);
    \coordinate  (12a) at ($(12)+(\drift,0)$);
    \coordinate  (13a) at ($(13)+(\drift,0)$);
    \coordinate  (14a) at ($(14)+(\drift,0)$);
    \coordinate  (15a) at ($(15)+(\drift,0)$);
    \coordinate  (16a) at ($(16)+(\drift,0)$);
    \coordinate  (17a) at ($(17)+(\drift,0)$);
    \coordinate  (18a) at ($(18)+(\drift,0)$);
    \coordinate  (19a) at ($(19)+(\drift,0)$);
    \coordinate  (20a) at ($(20)+(\drift,0)$);
    \coordinate  (21a) at ($(21)+(\drift,0)$);

    \foreach \point in {1a,2a,3a,5a,6a,7a,8a,9a,10a,11a,14a,15a,16a,17a}
    \fill [black] (\point) circle ({\rad});

    \foreach \point in {4a,12a,13a,18a,19a,20a,21a}
    \filldraw [fill=white] (\point) circle ({\rad});

    \coordinate  (c1a) at ($.5*(2a)+.5*(3a)$);
    \coordinate  (c2a) at ($.5*(18a)+.5*(21a)$);
    \coordinate  (c3a) at ($.5*(10a)+.5*(13a)$);
    \coordinate  (c4a) at ($.5*(15a)+.5*(16a)$);

    \draw (c1a) ellipse ({.5*\step} and {2*\step});
    \coordinate [label=left:$M_1$] (m1a) at ($(c1a)-(.5*\step,0)$);

    \draw (c2a) ellipse ({\step} and {\step});
    \coordinate [label=below:$M_2$] (m2a) at ($(c2a)-(0,\step)$);

    \draw (c3a) ellipse ({\step} and {\step});
    \coordinate [label=right:$M_3$] (m3a) at ($(c3a)+(\step,0)$);

    \draw (c4a) ellipse ({\step} and {\step});
    \coordinate [label=above:$M_4$] (m4a) at ($(c4a)+(0,\step)$);

    \draw [<->]
    ($(startarrow)+(0,-.8*\drift)$)--($(endarrow)+(0,-.8*\drift)$);

    \coordinate  (1b) at ($(1)+(\drift,-.8*\drift)$);
    \coordinate  (2b) at ($(2)+(\drift,-.8*\drift)$);
    \coordinate  (3b) at ($(3)+(\drift,-.8*\drift)$);
    \coordinate  (4b) at ($(4)+(\drift,-.8*\drift)$);
    \coordinate  (5b) at ($(5)+(\drift,-.8*\drift)$);
    \coordinate  (6b) at ($(6)+(\drift,-.8*\drift)$);
    \coordinate  (7b) at ($(7)+(\drift,-.8*\drift)$);
    \coordinate  (8b) at ($(8)+(\drift,-.8*\drift)$);
    \coordinate  (9b) at ($(9)+(\drift,-.8*\drift)$);
    \coordinate  (10b) at ($(10)+(\drift,-.8*\drift)$);
    \coordinate  (11b) at ($(11)+(\drift,-.8*\drift)$);
    \coordinate  (12b) at ($(12)+(\drift,-.8*\drift)$);
    \coordinate  (13b) at ($(13)+(\drift,-.8*\drift)$);
    \coordinate  (14b) at ($(14)+(\drift,-.8*\drift)$);
    \coordinate  (15b) at ($(15)+(\drift,-.8*\drift)$);
    \coordinate  (16b) at ($(16)+(\drift,-.8*\drift)$);
    \coordinate  (17b) at ($(17)+(\drift,-.8*\drift)$);
    \coordinate  (18b) at ($(18)+(\drift,-.8*\drift)$);
    \coordinate  (19b) at ($(19)+(\drift,-.8*\drift)$);
    \coordinate  (20b) at ($(20)+(\drift,-.8*\drift)$);
    \coordinate  (21b) at ($(21)+(\drift,-.8*\drift)$);

    \coordinate (k1) at ($.5*(4b)+.5*(10b)$);
    \coordinate (k2) at ($.5*(16b)+.5*(20b)$);

    \foreach \point in {1b,2b,3b,10b,11b,14b,15b,16b,17b}
    \fill [black] (\point) circle ({\rad});

    \foreach \point in {4b,12b,13b,18b,19b,20b,21b}
    \filldraw [fill=white] (\point) circle ({\rad});

    \foreach \point in {k1,k2}
    \filldraw[fill=white] ($(\point)+(-1.5pt,-1.5pt)$) rectangle ($(\point)+(1.5pt,1.5pt)$);

    \coordinate  (c1b) at ($.5*(2b)+.5*(3b)+(.25*\step,0)$);
    \coordinate  (c2b) at ($.5*(18b)+.5*(20b)+(.4*\step,0)$);
    \coordinate  (c3b) at ($.5*(10b)+.5*(12b)+(.3*\step,0)$);
    \coordinate  (c4b) at ($.5*(17b)+.5*(16b)-(.4*\step,0)$);

    \draw[rotate=5] (c1b) ellipse ({\step} and {2*\step});
    \coordinate [label=left:$M_1$] (m1b) at ($(c1b)-(\step,0)$);

    \draw (c2b) ellipse ({1.3*\step} and {\step});
    \coordinate [label=below:$M_2$] (m2b) at ($(c2b)-(0,\step)$);

    \draw (c3b) ellipse ({1.5*\step} and {\step});
    \coordinate [label=right:$M_3$] (m3b) at ($(c3b)+(1.5*\step,0)$);

    \draw (c4b) ellipse ({1.3*\step} and {\step});
    \coordinate [label=above:$M_4$] (m4b) at ($(c4b)+(0,\step)$);
  \end{tikzpicture}
  \caption{Switching for subcritical case.}
  \label{fig:switching-sub}
\end{figure}

Let
\begin{equation}
U(x_1,\dotsc, x_k)={(k\ell-x_1)^{k-1}
         \ell^{k\ell-x_1-k}}\binom{n}{2}^{\sum_{r\geq
           2}x_r}\beta^{k\ell-x_1},\lab{eq:U}
\end{equation}
for some constant $\beta$ to be determined later.  We prove that, for
every $k$-tuple $(x_1,\dotsc,x_k)\in\Xcal$ and for every ${\bf
  i}=(i_1,\dotsc, i_k)\in\Kcal'$, the number of inverse switchings
that can convert ${\bf i}$ to some $k$-tuples in
$\Kcal(x_1,\ldots,x_k)$ is at most $U(x_1,\ldots,x_k)$.  Note that it
is possible for some choices made in Steps~3 and~4, $M_{i_j}''$ may
not be an $\ell$-matching. But we only need an upper bound for the
number of inverse switchings in our case.  The number of integer
compositions of $k\ell-x_1$ into $k$ positive parts is
$\binom{k\ell-x_1-1}{k-1},$
        and so we have $\binom{k\ell-x_1-1}{k-1}$ choices for the vector
   $\rseq$ in Step~1. The number of choices in Step~2 is at most
   $\ell^{\sum_{1\le j\le k}(r_j-1)}$. In Step~3, we have at most
   $\binom{\binom{n}{2}}{|X|}$ choices for~$X$ and
   $\binom{|X|}{x_2,x_3,\dotsc, x_{k-1}}$ choices for the partition
   of~$X$. Now we bound the number of choices for $Q$ in
   step~4. This number equals the number of (simple) bipartite graphs
   $(X,[k])$ with the degrees as in Step 4 (note that the
   sum of the degrees of vertices in $X$ is $k\ell-x_1$).
   One can obtain a bipartite multigraph with the degrees as
   in Step~$4$ by the following procedure: replace each vertex in
   $Q$ with a set of points of size equal to its degree; add a
   perfect matching between the points arising from vertices in $X$
   to points arising from vertices in~$[k]$; for each vertex,
   contract the set of points arising from it. Note that each simple
   bipartite graph corresponds to ${\prod_{i=2}^k
     (i!)^{x_i}\prod_{j=1}^k r_j!}$ matchings. Moreover, there are
   $(k\ell-x_1)!$ choices for perfect matchings in the procedure.
   Thus, after restricting to counting only perfect matchings corresponding to simple bipartite graphs, the number of choices for $Q$ is at most
   \begin{equation*}
     \frac{(k\ell-x_1)!}{\prod_{i=2}^k (i!)^{x_i}\prod_{j=1}^k r_j!}
     \leq
     \frac{(k\ell-x_1)!}{\prod_{j=1}^k r_j!}
     \leq
     \frac{(k\ell-x_1)!}{\left(\floor{\frac{k\ell-x_1}{k}}!\right)^{k}}
     \leq
     \beta^{k\ell-x_1},
   \end{equation*}
   for a constant $\beta$ (depending only on $k$) by Stirling's approximation. Thus, the
 number of inverse switchings applicable on each ${\bf i}\in \Kcal'$ is
 at most
   \begin{equation}
     \begin{split}
       & \binom{k\ell-x_1-1}{k-1}
       \ell^{\sum_{1\le j\le k}(r_j-1)} \binom{\binom{n}{2}}{|X|}
       \binom{|X|}{x_2,x_3,\dotsc, x_{k-1}} \beta^{k\ell-x_1} \\
       &\leq {(k\ell-x_1)^{k-1}
         \ell^{k\ell-x_1-k}}\binom{n}{2}^{\sum_{r\geq
           2}x_r} \beta^{k\ell-x_1}=U(x_1,\dotsc, x_k),
  \end{split}
     \lab{eq:reverse-even}
\end{equation}
where we used that $\sum_{1\le j\le k}r_j = k\ell-x_1$. We will now proceed to bound the ratio
\begin{equation*}
  \frac{\sum_{\iseq\in
        \Kcal(x_1,\dotsc, x_k)} \cp(\iseq)}{\sum_{\iseq'\in \Kcal'}
      \cp(\iseq')},
\end{equation*}
for $(x_1,\dotsc, x_k)\in\Xcal$. Construct a bipartite multigraph $R$ with bipartition
$(\Kcal(x_1,\dotsc,x_k), \Kcal')$ such that $\iseq\in
\Kcal(x_1,\dotsc, x_k)$ and $\iseq'\in\Kcal'$ are adjacent if there
is a switching mapping $\iseq$ to $\iseq'$ and the number of edges
joining them is the number of such
switchings. By~\myeqref{eq:switch-even}, the degree of any vertex in
$\Kcal(x_1,\dotsc, x_k)$ is at least $L(x_1,\dotsc, x_k)$, and, on the
other hand, by~\myeqref{eq:reverse-even}, the degree of any vertex in
$\Kcal'$ is at most $U(x_1,\dotsc, x_k)$. Moreover, for any
$\iseq\in \Kcal(x_1,\dotsc, x_k)$ and $\iseq'\in\Kcal'$, we have that
\begin{equation}
  \frac{\cp(\iseq)}{\cp(\iseq')}
  \leq \frac{\prod_{C\in
      \Ccal(i)}p^{\cm_C}}{(p^{2\ell-1}-p^{2\ell})^{k/2}} =
  \frac{p^{\sum_{r=1}^\ell x_r}}{p^{k\ell-k/2}(1-p)^{k/2}}
  \eqdefinv p(x_1,\dotsc, x_k)
  \lab{eq:function-even-switching}.
\end{equation}
Let $N(\iseq)$ denote the set of neighbours of $\iseq$ in the bipartite multigraph $R$. For
each $\iseq\in\Kcal(x_1,\dotsc, x_k)$, we have then
\begin{equation*}
  \cp(\iseq)
  \leq
  \frac{p(x_1,\dotsc, x_k) \sum_{\iseq'\in N(\iseq)}
    \cp(\iseq')}{|N(\iseq)|}
  \leq\frac{p(x_1,\dotsc, x_k) \sum_{\iseq'\in N(\iseq)}
    \cp(\iseq')}{L(x_1,\dotsc,x_k)}
  \end{equation*}
and so
\begin{equation*}
  \begin{split}
    \sum_{\iseq \in\Kcal(x_1,\dotsc, x_k)} \cp(\iseq)
    &\leq
    \frac{p(x_1,\dotsc, x_k)}{L(x_1,\dotsc, x_k)}
    \sum_{\iseq\in \Kcal(x_1,\dotsc, x_k)} \sum_{\iseq'\in N(\iseq)}
      \cp(\iseq')
    \\
    &=
    \frac{p(x_1,\dotsc, x_k)}{L(x_1,\dotsc, x_k)}\sum_{\iseq'\in \Kcal'}
    \card{N(\iseq')}\cp(\iseq')
    \\
    &\leq
    \frac{p(x_1,\dotsc, x_k) U(x_1,\dotsc,x_k)}{L(x_1,\dotsc,x_k)} \sum_{\iseq'\in
    \Kcal'}\cp(\iseq').
  \end{split}
\end{equation*}
Recall
from~\eqref{eq:L} and~\eqref{eq:U} that
\begin{equation*}
  \begin{split}
    &L(x_1,\dotsc, x_k)={\left(\frac{1}{2}
        \binom{n}{2}\right)^{k\ell-x_1-k/2}},
    \\
    &U(x_1,\dotsc, x_k)= {(k\ell-x_1)^{k-1}
         \ell^{k\ell-x_1-k}}\binom{n}{2}^{\sum_{r\geq
           2}x_r}\beta^{k\ell-x_1}.
  \end{split}
\end{equation*}
Thus,
\begin{equation}
  \begin{split}
    \frac{\sum_{\iseq\in
      \Kcal(x_1,\dotsc, x_k)} \cp(\iseq)}{\sum_{\iseq'\in \Kcal'}
    \cp(\iseq')}
  &\leq
  \frac{U(x_1,\dotsc,x_k)}{L(x_1,\dotsc, x_k)}\cdot p(x_1,\dotsc, x_k)
  \\
&\le
  \frac{\ell^{k\ell-k-x_1}(k\ell-x_1)^{k-1}\beta^{k\ell-x_1}\binom{n}{2}^{\sum_{r\geq 2}x_r}}
    {\left(\frac{1}{2} \binom{n}{2} \right)^{k\ell-x_1-k/2}}
    \frac{p^{\sum_{r=1}^\ell x_r}}{p^{k\ell-k/2} (1-p)^{k/2}}\\
    &=
    \frac{\ell^{k\ell-k-x_1}(k\ell-x_1)^{k-1}\beta^{k\ell-x_1} \left(\binom{n}{2} p\right)^{\sum_{r\geq 2}x_r}}
    {\left(\frac{1}{2} \binom{n}{2}p \right)^{k\ell-x_1-k/2}}
    \frac{1}{(1-p)^{k/2}}
    \\
    &=
    O\left(
      \frac{(2\beta \ell)^{k\ell-k-x_1}(k\ell-x_1)^{k-1} }
    {\left( \binom{n}{2}p \right)^{k\ell-k/2-\sum_{i=1}^k x_i}}
  \right),
  \end{split}
\lab{eq:ratio-even}
   \end{equation}
   where the last equation holds because $1-p=\Omega(1)$ and $k$
     is fixed. Now we will bound
   \begin{equation*}
     \sum_{(x_1,\dotsc,x_k)\in \Xcal}\frac{\sum_{\iseq\in
         \Kcal(x_1,\dotsc, x_k)} \cp(\iseq)}{\sum_{\iseq'\in \Kcal'}
       \cp(\iseq')}.
   \end{equation*}
   We partition $\Xcal$ into two sets $\Xcal_1$ and $\Xcal_2$. Let
   $\Xcal_1$ be subset of $\Xcal$ such that $(x_1,\dotsc, x_k)\in
   \Xcal_1$ if $x_1 \leq k\ell-k-1$ and let $\Xcal_2$ be subset of
   $\Xcal$ such that $(x_1,\dotsc, x_k)\in
   \Xcal_2$ if $x_1 =
   k\ell-k$. Note that $\Xcal = \Xcal_1 \cup \Xcal_2$ since $x_1
     \leq k\ell-k$ for all $(x_1,\dotsc,x_k)\in \Xcal$.

   Recall that $\sum_{i=1}^k x_i = k\ell$ for all
   $(x_1,\dotsc,x_k)\in\Xcal$. Obviously, given the value of $x_1$,
 the number of nonnegative integral vectors $(x_2,\dotsc, x_k)$ such
 that $\sum_{r=2}^{k} r x_r = k\ell-x_1$ is
 $O(k^{(k\ell-x_1)/2})$, as $\sum_{ r=2}^{ k}x_r \le
 (k\ell-x_1)/2$. Together with~\myeqref{eq:ratio-even}, this implies
 that
\begin{equation*}
  \begin{split}
    \sum_{(x_1,\dotsc, x_k)\in\Xcal_1} &\frac{\sum_{\iseq\in
        \Kcal(x_1,\dotsc, x_k)} \cp(\iseq)}{\sum_{\iseq'\in \Kcal'}
      \cp(\iseq')}
    \\&=
    O\left(
      \sum_{x_1=0}^{k\ell-k-1}
    k^{(k\ell-x_1)/2}
      \frac{(2\beta \ell)^{k\ell-k-x_1}(k\ell-x_1)^{k-1} }
    {\left( \binom{n}{2}p \right)^{(k\ell-x_1-k)/2}}
   \right)\\
   &=
   O\left(
      \sum_{x_1=0}^{k\ell-k-1}
          \frac{(2k\beta \ell)^{k\ell-k-x_1}(k\ell-x_1)^{k-1} }
    {\left( \binom{n}{2}p \right)^{(k\ell-x_1-k)/2}}
   \right).
  \end{split}
\end{equation*}
Let
\begin{equation*}
  g(x_1) = \frac{(2\beta k\ell)^{k\ell-k-x_1}(k\ell-x_1)^{k-1}}
  {\left(\binom{n}{2}p \right)^{(k\ell-x_1)/2-k/2}}.
\end{equation*}
For $x_1 = k\ell-k-1$,
\begin{equation*}
  g(x_1)
  =
  \frac{(2 k\beta\ell)(k+1)^{k-1}}
  {\left(\binom{n}{2}p \right)^{1/2}}
=
O\left(
\frac{\ell}{n\sqrt{p}}
\right)
=
o(1),
\end{equation*}
since $\ell^2=o(n^2p)$. Moreover, for $x_1 < k\ell-k-1$, using
$\ell^2=o(n^2p)$ and $k\ell-x_1 > k+1$,
\begin{equation*}
  \frac{g(x_1)}{g(x_1+1)}
  =
  \frac{(2 k\beta\ell)}
  {\left(\binom{n}{2}p \right)^{1/2}}
  \cdot
  \left(\frac{k\ell-x_1}{k\ell-x_1-1}\right)^{k-1}
  \hspace{-5pt}
  \leq
  \frac{(2 k\beta\ell)}
  {\left(\binom{n}{2}p \right)^{1/2}}
  \cdot 2^{k-1}
=
O\left(
\frac{\ell}{n\sqrt{p}}
\right)
  = o(1),
\end{equation*}
since $\ell^2=o(n^2p)$ and both $k$ and $\beta$ are fixed constants. This shows that
\begin{equation}
  \lab{x1-small}
  \sum_{(x_1,\dotsc, x_k)\in\Xcal_1} \frac{\sum_{\iseq\in
      \Kcal(x_1,\dotsc, x_k)} \cp(\iseq)}{\sum_{\iseq'\in \Kcal'}
    \cp(\iseq')}
  = o(1).
\end{equation}
Now we deal with the case $(x_1,\dotsc, x_k)\in \Xcal_2$, i.e.,
$x_1=k\ell-k$. We have that $\sum_{i=2}^k ix_i = k$.  This implies
$\sum_{i=2}^{k} x_i \leq k/2$. But note that the only way
$\sum_{i=2}^{k} x_i = k/2$, would be $x_2=k/2$ and $x_i=0$ for $i\geq
3$ and $\Kcal(x_1,\dotsc, x_k)$ would be empty (since all such ${\bf
  i}$ are in $\Kcal'$) and there is nothing to prove in this
case. Thus, we can assume $\sum_{i=2}^{k} x_i \leq k/2-1$.  Using this
fact together with that the number of choices for $(x_2,\dotsc,
x_k)\in\setN$ such that $\sum_{i=2}^{k} i x_i = k$ is $O(1)$
and~\myeqref{eq:ratio-even},
\begin{equation}
\lab{x1-big}
  \begin{split}
    \sum_{(x_1,\dotsc, x_k)\in\Xcal_2}
    \frac{\sum_{\iseq\in\Kcal(k\ell-k,x_2,\dotsc, x_k)} \cp(\iseq)}{\sum_{\iseq'\in \Kcal'}
      \cp(\iseq')}
    &=
    O\left(
      \frac{(2\beta \ell)^{k\ell-k-x_1}(k\ell-x_1)^{k-1}}
    {\left( \binom{n}{2}p \right)^{k\ell-\sum_{i=1}^k x_i-k/2}}
  \right)
  \\
  &=
  O\left(
      \frac{1}
      {\binom{n}{2}p}
  \right)
 =
o(1),
\end{split}
\end{equation}
where in the last equality we use the fact that $n^2p\to\infty$.

Equation~\myeqref{x1-small} and Equation~\myeqref{x1-big} imply
\begin{equation*}
  \sum_{(x_1,\dotsc, x_k)\in\Xcal} \frac{\sum_{\iseq\in
      \Kcal(x_1,\dotsc, x_k)} \cp(\iseq)}{\sum_{\iseq'\in \Kcal'}
    \cp(\iseq')}
  = o(1).
\end{equation*}


Now we have completed the proof of Lemma~\ref{l:sub-significant} for
even integers $k\ge 2$. The proof for odd integers $k$ is analogous to
that for even $k$, with slightly more complication due to the treatment
of the single component that is a chained triple or a flower with $3$
petals that would appear in $\Ccal({\bf i})$, where ${\bf i}\in
\Kcal'$ for odd $k$. Thus, we will only give a sketch of the proof.
Slightly different from the case where $k$ is even, we split
$\Kcal\setminus \Kcal'$ into $\Kcal_0$ and $\Kcal_1$, where $\Kcal_1$
contains all $\Kcal(x_1,\ldots,x_k)$ with $x_1\le k(\ell-1)-1$,
whereas $\Kcal_0$ corresponds to $x_1=k(\ell-1)$. We also
partition $\Kcal'$ into $\Kcal'_0$ and $\Kcal'_1$ such that $\Kcal'_0$
contains all ${\bf i}\in \Kcal'$ for which the only component in
$\Ccal({\bf i})$ having three matchings is a flower with $3$ petals
(and thus $\Kcal'_1$ is with respect to the structure of a chained
triple). Obviously, if ${\bf i}\in \Kcal(x_1,\ldots,x_k)$ where
$x_1=k(\ell-1)$, then each component $C$ in $\cup_{1\le j\le k}
M_{i_j}$ is a flower and at least one flower has at least $3$ petals
because $k$ is odd. It is not difficult to verify (use switchings
  to remove all but one flower with 3 petals and form the
  other
  partial matchings into kissing pairs; the analysis for these
  switchings is similar but simpler compared with the above analysis
  for even $k$) that $\sum_{{\bf i}\in
  \Kcal_0}\cp(\iseq)=o(\sum_{{\bf i}\in \Kcal'_0}\cp(\iseq))$. Next we
show that
\begin{equation}
\sum_{{\bf i}\in \Kcal_1}\cp(\iseq)=o\left(\sum_{{\bf i}\in \Kcal'_1}\cp(\iseq)\right),\lab{triple}
\end{equation}
which will complete the proof of Lemma~\ref{l:sub-significant} for odd $k$.

Same as in the case of even $k$, we will define a switching from
$\Kcal(x_1,\dotsc, x_k)$ to~$\Kcal'_1$ where $x_1 \le k(\ell-1)-1$. As
the analysis is almost the same to the previous case, we omit the
calculations and just describe the switching operation and its
inverse. Let $(i_1,\dotsc, i_k)\in \Kcal(x_1,\dotsc, x_k)$.
Similarly to the even case, for each $j$, let $I_j$ denote
  the set of edges that $M_{i_j}$ shares with $\cup_{j'\neq
    j}M_{i_{j'}}$.
Note
  that, since $x_1\leq k(\ell-1)-1$, there exists a matching $M_{i_j}$
  such that $\card{I_{j}}\geq 2$.
 \begin{enumerate}
 \item \textit{(Delete shared edges)} For each $j\in[k]$, let
   $M_{i_j}' := M_{i_j}\setminus I_j$.
 \item \textit{(Choose a matching that will share two edges in the
     chained triple)} Choose $t\in[k]$ such that $|M'_{i_t}|\leq
   \ell-2$.
  \item \textit{(Obtain pairwise disjoint partial matchings each of
      size $(\ell -1)$ but one of size $(\ell-2)$)} For each
    $j\in[k]$, let $h_j=|I_j|-1$ if $j\neq t$ and $|I_j|-2$ if $j=t$.
    For each $j\in[k]$, choose edges $a_1^{(j)},\dotsc, a_{h_j}^{(j)}$,
    one after the other, so that
    $M_{i_j}'\cup\set{a_1^{(j)},\dotsc,a_{h_j}^{(j)}}$ is an
    $(\ell-1)$-matching (or $(\ell-2)$-matching if $j=t$) and
    $a_r^{(j)} \not\in\left(\bigcup_{j'=0}^{k} M_{i_{j'}}'\right) \cup
    \left(\bigcup_{j'<j}
      \set{a_1^{(j')},\dotsc,a_{h_{j'}}^{(j')}}\right)$ for every~$r$.
    Let $M_{i_j}'' \eqdef M_{i_j}' \cup
    \set{a_1^{(j)},\dotsc,a_{h_j}^{(j)}}$.
  \item \textit{(Build kissing pairs and chained triple)} Choose a
    graph $P$ on $[k]$ such that the degree of $t$ is~$2$ and the
    degree of any other vertex is~$1$.
  \item \textit{(Choose shared edges in kissing pairs and chained
      triple)}Let $e_1,\dotsc, e_{(k+1)/2}$ be an
    enumeration of the edges in~$P$. Let $t_1 < t_2$ be such that
    $e_{t_1}$ and $e_{t_2}$ are the edges incident to $t$. For each
    $r\in [(k+1)/2]\setminus\set{t_2}$, let $j$ and $j'$ denote the ends of $e_r$ and choose an
    edge $f_r$ such that $M_{i_j}''' := M_{i_j}''\cup\set{f_r}$ and
    $M_{i_{j'}}''' := M_{i_{j'}}''\cup\set{f_r}$ are matchings and
    $f_r\not\in \left(\bigcup_{a=1}^k M_{i_a}''\right)\cup
    \left(\bigcup_{b=1}^{r-1}f_b\right)$. For $r=t_2$, let $j$ denote
    the end of $e_r$ other than $t$ and choose an edge $f_r$ such that
    $M_{i_j}''' := M_{i_j}''\cup\set{f_r}$ and
    $M_{i_{t}}'''\cup\set{f_r}$ are matchings and $f_r\not\in
    \left(\bigcup_{a=1}^k M_{i_a}''\right)\cup \left(\bigcup_{b\neq
        t_2}f_b\right)$. Redefine $M_{i_{t}}'''$ by including~$f_r$.
  \item \textit{(Update the indices)} Let $\iseq'''$ be the new tuple
    such that for every $1\le j\le k$, $M_{i'''_j}=M'''_{i_j}$.
   \end{enumerate}
Similarly to the even case, the number of applicable switchings for any ${\bf i}\in \Kcal(x_1,\ldots,x_k)$ is at least
   \begin{equation*}
     \left(
       \frac{1}{2} \binom{n}{2}
     \right)^{\sum_{j} h_j + \frac{k-1}{2}+1}
     =
     \left(
       \frac{1}{2} \binom{n}{2}
     \right)^{k\ell-x_1-\frac{k+1}{2}}.
   \end{equation*}
Now we describe the inverse switching. Let $(i_1,\dotsc,
   i_k)\in\Kcal'_1$.
   \begin{enumerate}
   \item \textit{(Find the index of the matching in the chained triple
       that shares two edges)} Let $t\in[k]$ such that $|M_{i_t} \cap (\bigcup_{j\neq t}
     M_{i_j})|=2$.
   \item\textit{(Choose the number of edges shared by each matching)}
     Choose an integral vector $\rseq = (r_1,\dotsc, r_k)$ so that
     $\sum_{j=1}^{k} r_j = k\ell-x_1$ and $r_j\geq 1$ for every~$j$
     and $r_t\geq 2$.
   \item \textit{(Make room for shared edges)} For each $j\in[k]$,
     let $h_j = r_j-1$ if $j\neq t$ and $h_t = r_t-2$.  Choose $h_j$
     edges one by one without repetition in $M_{i_j}\setminus
     I_j$. Let $M_{i_j}'$ be $(\ell-r_j)$-matching obtained from
     $M_{i_j}$ by deleting these edges and the edges in $I_j$.
   \item \textit{(Choose shared edges)} Choose a set $X$ of
     $\sum_{i=2}^k x_i$ edges that are not in $\bigcup_{j=1}^k
     M_{i_j}'$. Partition $X$ into sets $X_2,\dotsc, X_k$ such that
     $|X_i|= x_i$ for every~$i$.
   \item \textit{(Assign shared edges)} Construct a bipartite graph
     $Q$ with bipartition $(X,[k])$ so that the degree of each $e\in
     X_i$ is $i$ and the degree of each $j\in[k]$ is
     $r_j$.  Let $M_{i_j}''$ be obtained from $M_{i_j}'$ by including
     the edges in $X$ that are adjacent to $j$ in~$Q$.
   \item \textit{(Update indices)} If each $M_{i_{j}}''$ is a $\ell$-matching, let $\iseq''$ be
     the new tuple such that for every $1\le j\le k$,
     $M_{i''_j}=M''_{i_j}$.
   \end{enumerate}
Similarly to the even case, the number of inverse switchings applicable to any ${\bf i}\in \Kcal'_1$ is at most
\begin{equation*}
     (k\ell-x_1)^{k-1}
     \ell^{k\ell-x_1-k-1}
     \binom{n}{2}^{\sum_{i=2}^k x_i}
     \beta^{k\ell-x_1},
   \end{equation*}
   for the same constant $\beta$ that is defined before.
   The same calculations as in the case of even $k$ (also by splitting
   the analysis into two cases $x_1=k(\ell-1)-1$ and $x_1\le
   k(\ell-1)-2$) give verification of~\eqref{triple}.  This completes
   the proof of Lemma~\ref{l:sub-significant}.
\end{proof}

\section{Proof of Theorem~\ref{t:super} and
  Theorem~\ref{t:nearperfect}}

\subsection{The log-normal paradigm}

Let $\S$ be a set of graphs on vertex set $[n]$ such that each graph
in $\S$ has $h$ edges (e.g.\ $\S$ is the set of $\ell$-matchings in
$K_{[n]}$). Let $s=|\S|$ and $X_n$ denote the number of graphs in $\S$
that are contained in a random graph ($\G(n,p)$ or $\G(n,m)$) as a
subgraph. Define
\begin{eqnarray}
 \mu_n&=&s\binom{N-h}{m-h}\Big/\binom{N}{m},\ \ \
\la_n=sp^h.\lab{mu-la}
\end{eqnarray}
Immediately we have
$$
\ex_{\G(n,m)} X_n=\mu_n,\ \ \ \ex_{\G(n,p)}X_n=\la_n.
$$
When $\limsup_{n\to\infty} h/m<1$ and
  $h^2=\Omega(m)$, we can further simplify $\mu_n$ and obtain
\begin{eqnarray}
  \mu_n&=&s\cdot\frac{[m]_{h}}{[N]_h}=s(m/N)^h\exp\left(-\frac{N-m}{mN}\frac{h^2}{2}+O(h^3/m^2)\right).\lab{mu}
\end{eqnarray}
 Given $i\ge 0$, let $F(i)=\{(G_1,G_2)\in
\S^2,\ |E(G_1)\cap E(G_2)|=i\}$ and let $f_i=|F(i)|$.

A slight generalisation of~\cite[Theorem 1]{G6}, with almost the same proof
(with only $f_j$ replaced by $f'_j$ and some equalities replaced by
asymptotic equalities), gives the following theorem.

\begin{thm}\lab{t:Gnm} Let $\mu_n$ be as in~\eqn{mu-la}.
  Suppose there is a sequence $(f'_j)_{j= 0}^h$ such that $f_j\sim
  f'_j$ uniformly for all $j\ge 0$. Let $r_j=f'_j/f'_{j-1}$ for all $1\le j\le
  h$.  Assume that $h^3=o(m^2)$, $h^2=\Omega(m)$, and, for
  $\rho(n)=h^2/m$ and some function $\gamma(n)$, the following
  conditions hold:
\begin{description}

\item{(a)} for all $K>0$ and for all $1\le j\le K\rho(n)$,
$$
r_j=\frac{h^2}{Nj}\paren[\bigg]{1+o\paren[\Big]{\frac{m}{h^2}}};
$$
\item{(b)} $r_j\le m/2N$ for all $4\rho(n)\le j\le \gamma(n)$;
\item{(c)} $t(n):=\sum_{j>\gamma(n)}f_j=o(\mu_n|\S|)$;
\end{description}
 Then, in $\G(n,m)$,
$$
X_n/\ex_{\G(n,m)}(X_n)\xrightarrow{p} 1
$$
as $n\to\infty$.

\end{thm}

The following theorem can also be found in~\cite[Theorem 3]{G6}.
\begin{thm}
  \lab{t:Gnp} Assume $h^3=o(p^2n^4)$. 
  Let
  $\beta_n=h\sqrt{(1-p)/pN}$ and $\la_n =
  \ex_{\G(n,p)}{X_n}$. Assume further that
  $\liminf_{n\to\infty}\beta_n>0$. If, for all $m=pN+O(\sqrt{pN})$, we
  have that $X_n/\ex_{\G(n,m)}(X_n)\xrightarrow{p} 1$, then, in
  $\G(n,p)$,
  $$
  \frac{\ln(e^{\beta_n^2/2} X_n/\la_n)}{\beta_n}\xrightarrow{d}
  \N(0,1) \ \ \mbox{as}\ n\to\infty,
  $$
  where $\N(0,1)$ is the standard normal distribution.
\end{thm}

We will apply~Theorems~\ref{t:Gnm} and~\ref{t:Gnp} with $\Scal = \Mcal
=\set{M_1,\dotsc, M_s}$, the set of all $\ell$-matchings of $K_{[n]}$.
For $0\leq i\leq \ell$, let $F(i)$ be the set of pairs of
$\ell$-matchings $(M, M')$ in $K_{[n]}$ such that $\card{M\cap M'} =
i$ and let $f_i = \card{F(i)}$. For any element $g=(M,M')$, let $n_0=n_0(g)$ denote the number of
vertices that are incident with neither $M$ nor $M'$; $n_1=n_1(g)$ the
number of vertices incident with exactly one of $M$ and $M'$;
$n_2=n_2(g)$ the number of vertices incident with both $M$ and $M'$.
Then we immediately have that $n = n_0+n_1+n_2$ and $4\ell = 2n_2 +
n_1$. This implies that
  \begin{equation}
  n_1 = 4\ell-2n_2, \ \  \mbox{and}\ \ n_0 = n-4\ell+n_2. \lab{n's}
  \end{equation}
  We will constantly use the relation~\eqn{n's} in the following
  proofs. Now we close this section by proving a lemma that will be
  used to verify condition (c) of Theorem~\ref{t:Gnm}
  in
  the proofs of Theorems~\ref{t:super} and~\ref{t:nearperfect}.
\begin{lemma}
  \lab{l:tail} Let $m=\omega(n)$. Suppose that $\ell^3 = o(m^2)$ and
  $\ell^2=\Omega(m)$. For $\delta > 4/5$, we have that $\sum_{i\ge
    \delta \ell} f_i = o(s\mu_n)$.
\end{lemma}
\begin{proof}
  Let $I = \lceil\delta\ell\rceil$ and $P = m/N$. We bound the
number of ways to choose a pair of matchings $(M, M')$ sharing at
least $I$ edges. There are $s$ choices for~$M$. The matching $M'$ has
to share at least $I$ edges with $M$ and the other edges of $M'$
cannot intersect these $I$ edges. Thus, we have at most $
{\binom{\ell}{I} \binom{n-2I}{2\ell-2I}} (2\ell-2I)!\Big/\paren[\big]{2^{(\ell-I)}
  (\ell-I)!}$ choices for $M'$. Thus, by~\eqref{mu},
  \begin{equation*}
    \begin{split}
      \frac{\sum_{i\geq I}f_i}{s\mu_n}
      &\leq
      \frac{\displaystyle{}\binom{\ell}{I} \binom{n-2I}{2\ell-2I}
        \frac{(2\ell-2I)!}{2^{\ell-I} (\ell-I)!}}
      {\mu_n}
      \\
      &=
      \frac{\displaystyle{\binom{\ell}{I} \binom{n-2I}{2\ell-2I}
        \frac{(2\ell-2I)!}{2^{\ell-I} (\ell-I)!}}}
      {\displaystyle{\binom{n}{2\ell}
          \frac{(2\ell)!}{2^\ell \ell!} P^{\ell}}}
      \exp\paren[\bigg]{
        \frac{(1-P)\ell^2}{2m} + O\paren[\Big]{\frac{\ell^3}{m^2}}
      }.
    \end{split}
  \end{equation*}
  Using Stirling's approximation for $(2I)!$ and $I!$ and the fact
  that $\binom{\ell}{I}\leq (\frac{e\ell}{I})^I$,
  $\binom{n}{2I}\geq (\frac{n}{2I})^{2I}$ and $n-2I\ge (1-\delta)n=\Omega(n)$, we have
  \begin{equation*}
    \begin{split}
      \frac{\displaystyle{\binom{\ell}{I} \binom{n-2I}{2\ell-2I}
          \frac{(2\ell-2I)!}{2^{\ell-I} (\ell-I)!}}}
      {\displaystyle{\binom{n}{2\ell} \frac{(2\ell)!}{2^\ell \ell!} P^{\ell}}}
      &=
      \frac{[\ell]_I^2 2^I}
        {I![n]_{2I} P^{\ell}}
      =\left(O\left(\frac{\ell^2}{(n-2I)^2I P^{\ell/I}}\right)\right)^I
      \\
      &=\left(O\left(\frac{1}{\sqrt{\ell} P^{\ell/I-1}}\right)\left(\frac{\ell^3}{n^4P^2}\right)^{1/2} \right)^I.
    \end{split}
  \end{equation*}
  We have that $\ell^3/(n^4P^2) = o(1)$ since $\ell^3=o(m^2)$ and
  \begin{equation*}
    \sqrt{\ell} P^{\ell/I-1}={\Omega(m^{1/4})P^{\delta^{-1}-1}} =\Omega( n^{1/4}P^{1/4}) \to \infty,
  \end{equation*}
  where the first equality holds because $\ell^2=\Omega(m)$ and the second equality because
  $\delta > 4/5$, and the last asymptotics holds
  because $nP\to \infty$ (since $m=\omega(n)$). Hence,
    \begin{equation*}
      \frac{\sum_{i\geq I}f_i}{s\mu_n}
      =
        \exp\left(\frac{(1-P)\ell^2}{2m} + O\paren[\Big]{\frac{\ell^3}{m^2}}-\omega(I)\right)
      =o(1)
    \end{equation*}
    since
    $\ell^3/m^2 = o(1)$ and $I = \Omega(\ell) = \Omega(\ell^2/m)$.
  \end{proof}

\subsection{Proof of Theorem~\ref{t:nearperfect}}

In this section, we apply Theorem~\ref{t:Gnm} and
  Theorem~\ref{t:Gnp} to prove Theorem~\ref{t:nearperfect}, which
  deals with near-perfect matchings. 
\begin{lemma}
  \lab{l:nearperfect-ratio} Let $0<\alpha<1$ be fixed and assume
  $n-2\ell=O(n^{\alpha})$. For any fixed $0<\delta<1$  and any $1\le i\le
  \delta\ell$,
  $$
  \frac{f_i}{f_{i-1}}=\frac{n^2}{8i\ell^2}(1+O(i/n+n^{\alpha-1})).
  $$
\end{lemma}
\begin{proof}
  We define the following switching (see
  Figure~\ref{fig:switching-nearperfect}). For any $g=(M,M')\in F(i)$,
  pick an edge $x\in M\cap M' $ and label the end vertices of $x$ by
  $1$ and $2$. Then pick edges $y\in M\setminus M' $ and $z\in M'
  \setminus M$ such that $y$ and $z$ are disjoint and label the end vertices
  of $y$ and $z$ by $3$, $4$ and $5$, $6$ respectively. Replace $x$
  and $y$ by $\{1,3\}$ and $\{2,4\}$ in~$M$ and replace $x$ and $z$ by
  $\{1,5\}$ and $\{2,6\}$ in~$M'$. This operation results in $g'\in
  F(i-1)$. The number of ways to perform such a switching is $2i\cdot
  2(\ell-i)\cdot 2(\ell-i+O(1))$, since the numbers of ways to choose
  $x$ and $y$ are $i$ and $\ell-i$ respectively, and the number of
  ways to choose $z$ is $\ell-i+O(1)$ where $O(1)$ accounts for the
  choices of $z$ such that $z$ and $y$ are not
  disjoint, and for each
  edge there are two ways to label its end vertices. The inverse
  switching can be described as follows. For any $g'=(Q,Q')\in
  F(i-1)$, pick a $2$-path in $Q\cup Q'$ and label the vertices as
  $3,1,5$ such that $\{3,1\}\in Q$ and $\{1,5\}\in Q'$. Pick another
  $2$-path in $Q\cup Q'$ and label the vertices as $4,2,6$ such that
  $\{4,2\}\in Q$, $\{2,6\}\in Q'$ and $\{3,4\}\notin Q$,
  $\{5,6\}\notin Q'$. Replace $\{3,1\}$ and $\{4,2\}$ by $\{3,4\}$ and
  $\{1,2\}$ in $Q$ and replace $\{1,5\}$ and $\{2,6\}$ by $\{1,2\}$
  and $\{5,6\}$ in $Q'$. This operation is applicable if and only if
  all six vertices $i$, $1\le i\le
  6$, are distinct. Recall that
  $n_2=n_2(g')$ denotes the number of vertices incident with both $Q$
  and $Q'$. By~\eqn{n's} and the assumption that
  $n-2\ell=O(n^{\alpha})$, it follows immediately that
  $n_2=n-O(n^{\alpha})$. There are $n_2-2(i-1)$ ways to choose vertex
  $1$ and then the vertices $3$ and $5$ are determined by the choice
  of vertex~$1$. The number of ways to choose $4,2,6$ is
  $n_2-2(i-1)-O(1)$, where $2(i-1)$ counts the number of vertices
  incident to edges in $Q\cap Q'$ and there are $O(1)$ ways to choose
  vertex $2$ so that either the six vertices are not all distinct, or
  $\{3,4\}\in Q'$, or $\{5,6\}\in Q$. Hence, the number of applicable
  inverse switchings for any $g'\in F(i-1)$ is
  $(n_2-O(i))^2=n^2(1+O(i/n+n^{\alpha-1}))$. Hence, for any $1\le i\le
  \delta\ell$,
  $$
  \frac{|F(i)|}{|F(i-1)|}
  =
  \frac{n^2}{8i(\ell-i)^2}(1+O(i/n+n^{\alpha-1}))
  =
  \frac{n^2}{8i\ell^2}(1+O(i/n+n^{\alpha-1})).\qedhere
  $$
\end{proof}

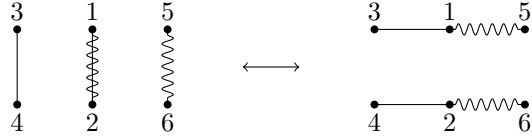
\begin{figure}[!h]
\centering
  \begin{tikzpicture}
    \coordinate [label=above:$1$] (1) at (0,0);
    \coordinate [label=below:$2$] (2) at (0,-1);
    \coordinate [label=above:$3$] (3) at (-1,0);
    \coordinate [label=below:$4$] (4) at (-1,-1);
    \coordinate [label=above:$5$] (5) at (1,0);
    \coordinate [label=below:$6$] (6) at (1,-1);
    \foreach \point in {1,2,3,4,5,6} \fill [black] (\point) circle
    (1.5pt);
    \draw (3) -- (4);
    \draw (1) -- (2);
    \draw [decorate,decoration={snake,segment length=1.5mm,
      amplitude=.75mm, pre length=2pt, post length=2pt}] (5) -- (6);
    \draw [decorate,decoration={snake,segment length=1.5mm,
      amplitude=.75mm, pre length=2pt, post length=2pt}] (1) -- (2);

    \coordinate (startarrow) at (2,-.5);
    \coordinate (endarrow) at (2.75,-.5);
    \draw [<->] (startarrow)--(endarrow);
    \coordinate [label=above:$1$] (1a) at (4.75,0);
    \coordinate [label=below:$2$] (2a) at (4.75,-1);
    \coordinate [label=above:$3$] (3a) at (3.75,0);
    \coordinate [label=below:$4$] (4a) at (3.75,-1);
    \coordinate [label=above:$5$] (5a) at (5.75,0);
    \coordinate [label=below:$6$] (6a) at (5.75,-1);
    \foreach \point in {1a,2a,3a,4a,5a,6a} \fill [black] (\point)
    circle (1.5pt);

    \draw (1a) -- (3a);
    \draw (2a) -- (4a);
    \draw [decorate,decoration={snake,segment length=1.5mm,
      amplitude=.75mm, pre length=2pt, post length=2pt}] (1a) -- (5a);
    \draw [decorate,decoration={snake,segment length=1.5mm,
      amplitude=.75mm, pre length=2pt, post length=2pt}] (2a) -- (6a);
  \end{tikzpicture}
  \caption{Switching to decrease the number of shared edges.}
  \label{fig:switching-nearperfect}
\end{figure}

\begin{proof}[Proof of Theorem~\ref{t:nearperfect}] For any
  $m=pN+O(\sqrt{pN})$, consider $\G(n,m)$. Apply Theorem~\ref{t:Gnm}
  with $h=\ell$ and $f'_j=f_j$ for all $0\le j\le \ell$. By our
  assumption, $pn^{1-\alpha}\to\infty$ as $n\to\infty$ and
  $\ell=n/2-O(n^{\alpha})$, where $\alpha>1/2$, which imply $p^2n\to\infty$. Hence, we have $\ell^3=o(m^2)$
  and $\ell^2=\Omega(m)$. Let $\gamma(n)=9\ell/10$. Conditions (a) and
  (b) are satisfied by Lemma~\ref{l:nearperfect-ratio} (by taking
  $\delta=9/10$) and the assumption that $pn^{1-\alpha}\to\infty$ as
  $n\to\infty$, and condition (c) is satisfied by Lemma~\ref{l:tail}.
  Hence, for all $m=pN+O(\sqrt{pN})$, we have
  $X_{n,\ell}/\ex_{\G(n,m)}(X_{n,\ell})\xrightarrow{p} 1$ as
  $n\to\infty$. We also have $1-p=\Omega(1)$ and $\ell=\Omega(n)$ by assumption, and so
  $\beta_{n,\ell}=\ell\sqrt{(1-p)/pN}=\Omega(1)$. Then
  Theorem~\ref{t:nearperfect} follows by Theorem~\ref{t:Gnp}.
\end{proof}

\subsection{Proof of Theorem~\ref{t:super}}

In this section, we apply Theorem~\ref{t:Gnm} and
  Theorem~\ref{t:Gnp} to prove Theorem~\ref{t:super}. We assume the
  hypotheses in Theorem~\ref{t:super}: $np
  \to\infty$, $1-p = \Omega(1)$, $\ell = \Omega(n\sqrt{p})$, $\ell
  \leq n/2 - n^{\alpha}$ with $\alpha \in(7/8, 1)$ being fixed, and
  $\ell^3 = o(n^4p^2)$.

For $0\leq i\leq \ell$ and $0\leq n_2\leq 2\ell$, let $F(i,n_2)$ be
the set of pairs $(M, M')\in F(i)$ such that $|V(M)\cap V(M')| = n_2$
and let $f(i, n_2) = |F(i, n_2)|$. Let $\delta=9/10$. For
$0\leq i\leq \delta \ell$, define
\begin{equation}\lab{z}
  \begin{split}
    z(i) &= \frac{4(\ell-i)^2}{n-2i},\\
    f_i' &=
    \sqrt{\pi}
    \left(\frac{1}{2z(i)} +
      \frac{1}{2\ell-z(i)-2i}+\frac{1}{2(n-4\ell+z(i)+2i)}\right)^{-1/2}
    f(i, z(i)+2i).
  \end{split}
\end{equation}
For $\delta \ell < i\leq \ell$, let $f'_i = f_i$.

First we prove that $f_i$ and $f'_i$ are asymptotically equal. In many
places, we ignore the floor sign if a certain variable is required to
be integral (e.g.\ the number of edges) but the error caused by
ignoring it in the analysis is negligible.

\begin{lemma}
  \lab{l:super-formula}For $i \leq \delta \ell$, we have that $f_i =
  f'_i (1 + o(1))$, uniformly for $i$. Moreover, $f(i, z(i)+2i+k)/f(i,
  z(i)+2i)= 1+ O\paren[\big]{1/z(i)+n/
    (\ell(n-2\ell))+n/(n-2\ell)^2}$ for $k = O(1)$,
  uniformly for $i$.
\end{lemma}
\begin{proof}
  We define the following switching (see
  Figure~\ref{fig:swiching-super}). Given a pair of matchings
  $(M,M')\in F(i,n_2)$, choose a vertex $v$ saturated by both $M, M'$
  with distinct edges, say $av\in M, bv \in M'$, choose a vertex $u$
  not saturated by neither matching, delete $av$ from $M$ and add $au$
  to $M$. The new pair of matchings is in $F(i, n_2-1)$. Note that
  there are $(n_2-2i) n_0$ ways of performing this switching.

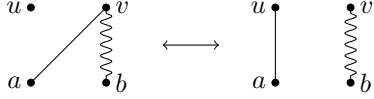
\begin{figure}[!h]
\centering
  \begin{tikzpicture}
    \coordinate [label=left:$u$] (u) at (0,0);
    \coordinate [label=left:$a$] (a) at (0,-1);
    \coordinate [label=right:$v$] (v) at (1,0);
    \coordinate [label=right:$b$] (b) at (1,-1);
    \foreach \point in {u,a,v,b} \fill [black] (\point) circle (1.5pt);
    \draw (u) -- (a);
    \draw [decorate,decoration={snake,segment length=1.5mm,
      amplitude=.75mm, pre length=2pt, post length=2pt}]  (v) -- (b);

    \coordinate (startarrow) at (-3.5+2,-.5);
    \coordinate (endarrow) at (-3.5+2.75,-.5);
    \draw [<->] (startarrow)--(endarrow);

    \coordinate [label=left:$u$] (u2) at (-7+3.75,0);
    \coordinate [label=left:$a$] (a2) at (-7+3.75,-1);
    \coordinate [label=right:$v$] (v2) at (-7+4.75,0);
    \coordinate [label=right:$b$] (b2) at (-7+4.75,-1);
    \foreach \point in {u2,a2,v2,b2} \fill [black] (\point) circle (1.5pt);
    \draw (a2) -- (v2);
    \draw [decorate,decoration={snake,segment length=1.5mm,
      amplitude=.75mm, pre length=2pt, post length=2pt}]  (v2) -- (b2);
  \end{tikzpicture}
  \caption{Switching to decrease $n_2$.}
  \label{fig:swiching-super}
\end{figure}

  The inverse switching is described as follows. Given a pair of
  matchings $(M, M')\in F(i, n_2-1)$, choose vertices $u$ covered by
  $M$ but not by $M'$ and $v$ covered by $M'$ but not by $M$ such that
  the edge $au\in M$ and the edge $bv\in M'$ satisfy $a\neq b$. Delete
  $au$ from $M$ and add $av$ to $M$. The number of vertices that are
  saturated only by $M$ is $n_1/2$ and so is the number of vertices
  that are saturated only by $M'$. Hence, there are $n_1 (n_1 -
  O(1))/4$ ways of doing the switching, where the $O(1)$ accounts for
  the choices of $v$ such that $a=b$, given the choice of $u$.

   By~\eqn{n's},
  \begin{equation*}
    \frac{f(i, n_2)}{f(i, n_2-1)}=\frac{n_1 (n_1 - O(1))/4}{(n_2-2i) n_0}
    =\frac{(2\ell-n_2)^2}{(n_2-2i)(n-4\ell+n_2)}
    \left(1 +
      O\left(\frac{1}{2\ell-n_2}\right)\right).
  \end{equation*}
For $n_2 = z(i)+2i$, we have that this ratio (ignoring the
  error term) is $1$. Thus, for $k < z(i)^{\alpha'/3}$ with $\alpha'
  \in (1.5,2)$ satisfying $2\alpha > 1+\alpha'/2$,
  \begin{align*}
    &\frac{f(i, z(i)+2i+k)}{f(i,z(i)+2i)} = \prod_{j=1}^{k}
    \frac{f(i, z(i)+2i+j)}{f(i, z(i)+2i+j-1)}
    \\& = \prod_{j=1}^{k}
    \frac{(2\ell-z(i)-j-2i)^2}{(z(i)+j)(n-4\ell+z(i)+j+2i)} \left(1 +
      O\left(\frac{1}{2\ell-z(i)-j-2i}\right)\right)
    \\
    &= \prod_{j=1}^{k} \frac{(2\ell-z(i)-2i)^2}{z(i)(n-4\ell+z(i)+2i)}
    \frac{\left(1 -\frac{j}{2\ell-z(i)-2i}\right)^2\left(1 +
      O\left(\frac{1}{2\ell-z(i)-j-2i}\right)\right)}
    {\left(1+\frac{j}{z(i)}\right)\left(1 +
        \frac{j}{n-4\ell+z(i)+2i}\right)} \\
     &= \prod_{j=1}^{k} \exp\Bigg(
    -\frac{2j}{2\ell-z(i)-2i} - \frac{j}{z(i)} -
    \frac{j}{n-4\ell+z(i)+2i}+
    \\
    &+ O\left( \frac{j^2}{(2\ell-z(i)-2i)^2} +
      \frac{j^2}{z(i)^2}+
      \frac{j^2}{(n-4\ell+z(i)+2i)^2}+ \frac{1}{2\ell-z(i)-j-2i}\right)\Bigg)
    \\
    &= \exp\Bigg( -\frac{k^2}{2\ell-z(i)-2i} - \frac{k^2}{2z(i)} -
    \frac{k^2}{2(n-4\ell+z(i)+2i)}+
    \\
       &+
       O\left(\frac{k}{z(i)}+
      \frac{k}{n-4\ell+z(i)+2i}\right)
    \\& + O\left(\frac{k^3}{(2\ell-z(i)-2i)^2} +
      \frac{k^3}{z(i)^2}+
      \frac{k^3}{(n-4\ell+z(i)+2i)^2}+ \frac{k}{2\ell-z(i)-k-2i}\right) \Bigg).
  \end{align*}

We have that the derivative of $z(i)$ with respect to $i$ is
$-8(\ell-i)(n-\ell-i)/(n-2i)^2$. Since $n-\ell-i \geq n-2\ell \geq 0$,
this implies that $z(i)\ge z(\floor{\delta \ell})$ for all
$i\le\delta\ell$. Moreover, $z(\floor{\delta \ell}) = \Omega(\ell^2/n)
= \omega(1)$ because $\ell=\omega(\sqrt{n})$ (this holds
since $\ell^2 = \Omega(n^2p)$ and $np = \omega(1)$). Thus, using $k <
z(i)^{\alpha'/3}$ and $\alpha' < 2$, we have that
\begin{equation*}
  \frac{k^3}{z(i)^2}
  <
  \frac{1}{z(i)^{2-\alpha'}}
  \leq
  \frac{1}{z(\floor{\delta \ell})^{2-\alpha'}}
  =
  o(1),
\end{equation*}
and so $k/z(i) = o(1)$ as well.

Using~$k<z(i)^{\alpha'/3}$,
\begin{equation*}
  \begin{split}
    \frac{k^3}{(2\ell-z(i)-2i)^2}&
    <\frac{z(i)^{\alpha'}}{(2\ell-z(i)-2i)^2}
    =O\left(\frac{(\ell-i)^{2\alpha'-2}(n-2i)^{2-\alpha'}}{(n-2\ell)^2}\right)
    \\
    &=
    O\left(n^{\alpha'-2\alpha}\right) = o(1),
  \end{split}
\end{equation*}
where the first equation is by~\eqn{z} and the fact that
$2\alpha'-2,2-\alpha'>0$, the second equation is by noting that
$n-2\ell= \Omega(n^{\alpha})$ and  $\ell-i,n-2i\le n$, and
the last equation holds because $2\alpha>1+\alpha'/2$ and $\alpha'<2$,
implying that $\alpha'<2\alpha$. We also have that
\begin{equation*}
  \begin{split}
    \frac{k^3}{(n-4\ell+z(i)+2i)^2}
    &<
    \frac{z(i)^{\alpha'}}{(n-4\ell+z(i)+2i)^2}
    =
    \frac{4^{\alpha'}(\ell-i)^{2\alpha'}
      (n-2i)^2}{
      (n-2i)^{\alpha'}(n-2\ell)^4}
    \\
    &=
    O\left(\frac{\ell^{2\alpha'}(n-2i)^{2-\alpha'}}{(n-2\ell)^{4}}\right)
    = O\left(n^{2+\alpha'-4\alpha}\right)=o(1),
  \end{split}
\end{equation*}
where the first equation is by~\eqn{z}, the second equation follows
from $\alpha',2-\alpha' > 0$, the third equation is by $n-2i, \ell\leq
n$ and $n-2\ell = \Omega(n^\alpha)$, and the last equation holds
because $2\alpha > 1 + \alpha'/2$. Moreover, $
  k^3/(n-4\ell+z(i)+2i)^2=o(1)$ implies that $k/(n-4\ell+z(i)+2i)=o(1)$ as well. We have that
\begin{equation*}
  \begin{split}
    \frac{z(i)^{\alpha'/3}}{2\ell-z(i)-2i}
    &=
    \frac{4^{\alpha'/3}(\ell-i)^{2\alpha'/3}}{(n-2i)^{\alpha'/3}}\frac{n-2i}{2(n-2\ell)(\ell-i)}
    =
    O\left(
      \frac{n^{2\alpha'/3-1}n^{1-\alpha'/3}}{n^{\alpha}}
    \right)
    \\&=
    O\left(
      {n^{-\alpha+\alpha'/3}}
    \right)
    =
    o(1),
  \end{split}
\end{equation*}
where the first equation is by~\eqn{z}, the second equation follows
from $\ell-i, n-2i\leq n$, $\alpha'\in(3/2,2)$ and $n-2\ell =
\Omega(n^\alpha)$, and the last equation follows from the fact that $2
\alpha>1+\alpha'/2$, which implies $\alpha > \alpha'/3$.

Thus,  for any $k < z(i)^{\alpha'/3}$,
\begin{equation*}
  \begin{split}
    \frac{f(i,z(i)+2i+k)}{f(i,z(i)+2i)} &\sim \exp\left( -\frac{k^2}{2\ell-z(i)-2i} - \frac{k^2}{2z(i)} -
      \frac{k^2}{2(n-4\ell+z(i)+2i)}\right).
  \end{split}
\end{equation*}

Then, for $k = O(1)$, 
 \begin{equation*}
  \frac{f(i, z(i)+2i+k)}{f(i, z(i)+2i)}
  =
  \exp\left(
    O\left(-\frac{1}{2\ell-z(i)-2i}
      - \frac{1}{2z(i)}
      - \frac{1}{2(n-4\ell+z(i)+2i)}\right)\right).
\end{equation*}
We have that $2\ell -z(i)-2i = \Omega(\ell(n-2\ell)/ n )$ and
$n-4\ell+z(i)+2i = \Omega((n-2\ell)^2/n)$.  This implies that
$f(i,z(i)+2i+k)/f(i, z(i)+2i) = 1+ O(1 / z(i)+n/
(\ell(n-2\ell)+n/(n-2\ell)^2)$ for $k=O(1)$.

Note that the ratio between consecutive terms is decreasing as $n_2$
increases (moreover we can ignore the error in the ratio because we
only need an upper bound now). If $k' = z^{\alpha''/3}$ with $1.5 <
\alpha'' < \alpha'$, then
\begin{equation*}
  \begin{split}
    \frac{\sum_{j\geq k'} f(i, z(i)+2i+j)}{f(i, z(i)+2i)}
    &\leq (1+o(1)) \sum_{j\geq k'}\exp\left(-\frac{j^2}{2z(i)}\right)
    \\
    &\leq  (1+o(1)) \sum_{j\geq k'}\exp\left(-\frac{j k'}{2z(i)}\right)
    \\
    &\leq
    (1+o(1)) \frac{\exp\left(-\frac{(k')^2}{2z(i)}\right)}
    {1 - \exp\left(-\frac{k'}{2z(i)} \right)}
    \\
    &\sim
    \exp\left(-\frac{(k')^2}{2z(i)}-\ln\left(\frac{z(i)}{k'}\right)\right)
    \\
    &=o(1),
  \end{split}
\end{equation*}
since $z(i)^{1/2}< k'< z(i)^{2/3}$ and $z(i)=\omega(1)$. Thus, we can
ignore the terms with $z(i)+2i+j$ with $j > k'$. By similar computations,
we have
\begin{equation*}
  \frac{f(i, z(i)+2i-k)}{f(i, z(i)+2i)}
  \sim  \exp\left(
    -\frac{k^2}{2\ell-z(i)-2i}
    - \frac{k^2}{2z(i)}
    - \frac{k^2}{2(n-4\ell+z(i)+2i)}\right)
\end{equation*}
and the lower tail can be bounded in the same manner. Thus,
\begin{equation*}
  \begin{split}
    f_i &\sim \sum_{k=-z(i)^\alpha}^{z(i)^\alpha} f(i,z(i)+2i)\times
    \\
    &\quad\qquad \times\exp\left( -\frac{k^2}{2\ell-z(i)-2i} - \frac{k^2}{2z(i)} -
      \frac{k^2}{2(n-4\ell+z(i)+2i)}\right)
    \\
    &\sim f(i, z(i)+2i) \sqrt{n}\times
    \\
    &\quad\qquad\times\int_{y = -\infty}^{\infty}
    \hspace{-15pt}\exp\left( -y^2
    \left(\frac{n}{2\ell-z(i)-2i} + \frac{n}{2z(i)} +
        \frac{n}{2(n-4\ell+z(i)+2i)}\right)\right) dy
    \\
    &\sim \sqrt{\pi}\left({\frac{1}{2z(i)} +
        \frac{1}{2\ell-z(i)-2i}+\frac{1}{2(n-4\ell+z(i)+2i)}}\right)^{-1/2}  f(i,
    z(i)+2i) \\
    &= f'_i.\qedhere
  \end{split}
\end{equation*}
\end{proof}

In the next lemma, we compute the ratio $f'_i/f'_{i-1}$.

\begin{lemma}
    \lab{l:super-ratio} Suppose that $i \leq \delta \ell$. Then
  \begin{equation*}
    \frac{f'_i}{f'_{i-1}}
    =
    \frac{z(i)^2}{8 i (\ell-i)^2}
    \left(1  + O\left(\frac{n}{\ell(n-2\ell)}\right) +
        O\left(\frac{n}{(n-2\ell)^2}\right) +O\left(\frac{1}{z(i)}\right)\right).
  \end{equation*}
\end{lemma}
\begin{proof}
  By~\eqn{z}, for $i\leq \delta \ell$, we have
\begin{align*}
  \frac{f'_i}{f'_{i-1}}
  =
  &\sqrt{\frac{\displaystyle{\frac{1}{2z(i-1)} +
      \frac{1}{2\ell-z(i-1)-2i+2}+\frac{1}{2(n-4\ell+z(i-1)+2i-2)}}}{\displaystyle{\frac{1}{2z(i)} +
      \frac{1}{2\ell-z(i)-2i}+\frac{1}{2(n-4\ell+z(i)+2i)}}}}
  \\
  &\cdot
  \frac{f(i-1,z(i)+2i)}{f(i-1,z(i-1)+2i-2)}
  \cdot
  \frac{f(i,z(i)+2i)}{f(i-1,z(i)+2i)}
\end{align*}
We will analyse each of these three ratios separately. The square of the first ratio (the expression can be easily simplified using Maple) equals
\begin{equation*}
  \frac{(\ell-i)^2(n- 2i + 2)^3}
  {(\ell-i+1)^2(n-2i)^3}
  =
  1+O\paren[\Big]{\frac{1}{\ell-i}}+O\paren[\Big]{\frac{1}{n-2i}}.
\end{equation*}
We have that $1/(\ell-i) = O(1/\ell)$ since $i\leq \delta \ell$ with
$\delta < 1$ and similarly, $1/(n-2i)=O(1/\ell)$.
Thus, the first ratio is $1 + O(1/\ell)$. Next, we analyse the
second ratio. Using that $\ell^2/n \to \infty$, it follows easily
that $z(i-1)-z(i) = O(1)$. Thus, we
have that the second ratio is $1+ O(1/z(i)+n/ (\ell(n-2\ell)) +n/(n-2\ell)^2)$ by
Lemma~\ref{l:super-formula}.

Finally, we analyse the last ratio. We use the same switching in the
proof of Lemma~\ref{l:nearperfect-ratio} to analyse the ratio
$f(i,n_2)/f(i-1,n_2)$, where $n_2 = z(i)+2i$. As it was shown that the
number of ways to perform a switching is
$8i(\ell-i)(\ell-i+O(1))=8i(\ell-i)^2(1+O(1/\ell))$, and the number of
ways to perform an inverse switching is
$(n_2-2(i-1))(n_2-2(i-1)+O(1))=z(i-1)^2(1+O(1/z(i)))$. Thus, we have
\begin{equation*}
{\frac{f(i,z(i)+2i)}{f(i-1,z(i)+2i)}}
  =
  \frac{z(i)^2}{8i(\ell-i)^2}
  \left(1+ O\left(\frac{1}{\ell}\right) +
    O\left(\frac{1}{z(i)}\right)\right).\qedhere
\end{equation*}
\end{proof}

\begin{proof}[Proof of Theorem~\ref{t:super}]
  For any $m = pN + O(\sqrt{Np})$, consider $\G(n,m)$. Apply
  Theorem~\ref{t:Gnm} with $h=\ell$ and $(f'_j)_{j=1}^{\ell}$. By
  Lemma~\ref{l:super-formula}, we have that $f_j\sim f_j'$ for all
  $1\leq j\leq \ell$. By our assumptions on~$p$ and~$\ell$, it is
  straightforward to verify that $\ell^3=o(m^2)$ and
  $\ell^2=\Omega(m)$. Next, we show that the conditions (a)--(c) in
  Theorem~\ref{t:Gnm} hold for
  $\gamma(n):=\delta\ell=9\ell/10$. By
  Lemma~\ref{l:super-ratio}, for $1\leq j \leq K \ell^2/m$, we have
  that
  \begin{equation*}
    \begin{split}
      r_j
      &= \frac{z(j)^2}{8j(\ell-j)^2}
      \left(1  + O\left(\frac{n}{\ell(n-2\ell)}\right) +
        O\left(\frac{n}{(n-2\ell)^2}\right)
        +O\left(\frac{1}{z(j)}\right)\right)
      \\
      &=
      \frac{\ell^2}{Nj}
      \left(1  + O\left(\frac{j}{\ell}\right) + O\left(\frac{n}{\ell(n-2\ell)}\right) +
      O\left(\frac{n}{(n-2\ell)^2}\right)
      +O\left(\frac{1}{z(j)}\right)\right).
    \end{split}
  \end{equation*}
  We have that $\ell=\Omega(n\sqrt{p}) =
  \Omega(\sqrt{n}\cdot\sqrt{np}) = \omega(\sqrt{n})$ since $np\to
  \infty$. Using this together with $\ell^3 = o(n^4p^2)$, we obtain
  \begin{equation*}
    n^{1+\alpha} p
    =
    \omega\left(n^{1+\alpha} \frac{\ell^{3/2}}{n^2} \right)
    =
    \ell\omega\left( \frac{\ell^{1/2}}{n^{1-\alpha}} \right)
    =
    \ell\omega(n^{1/8})
  \end{equation*}
  and
  \begin{equation*}
    n^{1+2\alpha} p
    =  \omega\left(n^{1+2\alpha} \frac{\ell^{3/2}}{n^2} \right)
    =\ell^2\omega\left(\frac{n^{-1+2\alpha}}{\ell^{1/2}} \right)
    =
    \ell^2\omega(n^{1/4}).
  \end{equation*}
  This implies $\ell = o(n^{1+\alpha} p)$ and $\ell^2 =
  o(n^{1+2\alpha} p)$. Thus, we have that
  \begin{equation*}
    \frac{j}{\ell}
    \leq \frac{K\ell}{m}    =
    O\left(
      \frac{\ell}{n^2 p}
    \right)
    =
    O\left(
      \frac{n^2 p}{\ell^2}
    \right)
    \cdot
    O\left(
      \frac{\ell^3}{n^4p^2}
    \right)
    =
    o \left(
      \frac{n^2 p}{\ell^2}
    \right);
  \end{equation*}
  and
  \begin{equation*}
    \begin{split}
      \frac{n}{\ell(n-2\ell)}
      = O\left(
        \frac{n^{1-\alpha}}{\ell}
      \right)
      =
      O\left(
        \frac{n^2p}{\ell^2}
      \right)
      \cdot
      O\left(
        \frac{\ell}{n^{1+\alpha}p}
      \right)
      =
      o\left(
        \frac{n^2p}{\ell^2}
      \right);
        \end{split}
  \end{equation*}
  and
  \begin{equation*}
    \begin{split}
      \frac{n}{(n-2\ell)^2}
            = O\left(
         n^{1-2\alpha}
      \right)
      =
      O\left(
        \frac{n^2p}{\ell^2}
      \right)
      \cdot
      O\left(
        \frac{\ell^2}{n^{1+2\alpha}p}
      \right)
      =
      o\left(
        \frac{n^2p}{\ell^2}
      \right);
    \end{split}
  \end{equation*}
  and
  \begin{equation*}
    \begin{split}
      \frac{1}{z(j)}
      =
      O\left(\frac{n}{\ell^2}\right)
      =
          O\left(
        \frac{n^2p}{\ell^2}
      \right)
      \cdot
      O\left(
        \frac{1}{np}
      \right)
      =
      o\left(
        \frac{n^2p}{\ell^2}
      \right).
    \end{split}
  \end{equation*}
  Thus, condition (a) holds. Now we will check condition (b). We have
  that for, $4\ell^2/m \leq j\leq \delta \ell$,
  \begin{equation*}
    r_j
    \sim\frac{z(j)^2}{8j(\ell-j)^2}
    =
    \frac{2(\ell-j)^2}{j(n-2j)^2}.
  \end{equation*}
  By computing the derivative of the RHS with respect to $j$ and using $n\geq 2\ell$, it
  is easy to see that the derivative is negative. At $j = 4\ell^2/m$,
  using $\ell^2/(n^3p) = o(1)$,
  \begin{equation*}
    \frac{2(\ell-j)^2}{j(n-2j)^2}
    \leq
    \frac{m}{2(n-2j)^2}
    =
    \frac{m}{2n^2 (1+o(1))}
    \leq
    \frac{m}{4N}(1+o(1)).
    \end{equation*}
    So condition (b) holds. Condition (c) holds by Lemma~\ref{l:tail}
    (with $\delta=9/10$). Hence, we have that
    $X_{n,\ell}/\ex_{\G(n,m)}(X_{n,\ell})\xrightarrow{p} 1$ by
    Theorem~\ref{t:Gnm}. Since $1-p=\Omega(1)$ by assumption, we have
    $\beta_{n,\ell}=\Omega(\ell/\sqrt{p}n)=\Omega(1)$. Then
    Theorem~\ref{t:super} follows by Theorem~\ref{t:Gnp}.
\end{proof}


\end{document}